\documentclass[11pt,twoside]{amsart}

\usepackage{xspace}
\usepackage{amssymb}
\usepackage{bbm}
\usepackage{graphicx}
\usepackage{url}
\usepackage{pifont}
\usepackage{enumerate}
\usepackage[matrix,arrow]{xy}
\usepackage{tikz}

\numberwithin{equation}{section}

\newcommand{\mi}{\bbi\xspace}

\DeclareMathSymbol{\varnothing}{\mathord}{AMSb}{"3F}

\DeclareMathOperator{\Res}{Res}
\DeclareMathOperator{\tr}{tr}
\DeclareMathOperator{\ord}{ord}

\newcommand{\bbC}{\mathbb{C}}

\newcommand{\bbE}{\mathbb{E}}

\newcommand{\bbH}{\mathbb{H}}

\newcommand{\bbN}{\mathbb{N}}

\newcommand{\bbR}{\mathbb{R}}
\newcommand{\bbS}{\mathbb{S}}

\newcommand{\bbZ}{\mathbb{Z}}

\newcommand{\R}{\mathbb{R}}
\newcommand{\C}{\mathbb{C}}
\newcommand{\one}{\mathbbm{1}}

\newcommand{\calO}{\mathcal{O}}

\newcommand{\CPone}{\bbC\mathrm{P}^1}

\theoremstyle{plain}
\newtheorem{theorem}{Theorem}[section]
\newtheorem*{theorem*}{Theorem}

\newtheorem*{corollary*}{Corollary}
\newtheorem{proposition}[theorem]{Proposition}
\newtheorem*{proposition*}{Proposition}
\newtheorem{lemma}[theorem]{Lemma}
\newtheorem*{lemma*}{Lemma}

\newtheorem*{example*}{Example}

\newtheorem*{definition*}{Definition}

\newtheorem*{notation*}{Notation}
\newtheorem{remark}[theorem]{Remark}
\newtheorem*{remark*}{Remark}

\newcommand{\bbi}{\mathbbm{i}}

\title{On closed finite gap curves in spaceforms II}
\author{S.~Klein}
\author{M.~Kilian}

\address{S.~Klein, Institut f\"ur Mathematik, Universit\"at Mannheim, Germany.}
\email{s.klein@math.uni-mannheim.de}
\address{M.~Kilian, Department of Mathematics,
University College Cork, Ireland.}
\email{m.kilian@ucc.ie}

\thanks{{\it Mathematics Subject Classification.} 53A04, 37K15. January 10, 2019.}

\begin{document}
\begin{abstract} 
We prove that the set of closed finite gap curves in hyperbolic 3-space $\bbH^3$ is $W^{2,2}$-dense in the Sobolev space of all closed $W^{2,2}$-curves in $\bbH^3$. We also show that the set of closed finite gap curves in any 2-dimensional space form $\bbE^2$ is $W^{2,2}$-dense in the Sobolev space of all closed $W^{2,2}$-curves in $\bbE^2$.
\end{abstract}
\maketitle


\section*{Introduction} 
The (self-focusing) non-linear Schr\"odinger (NLS) hierarchy is an infinite hierarchy of commuting energy functionals on the space of curves in 3-dimensional space forms. A curve is said to be of finite gap if its complex curvature function (in the sense of Hasimoto \cite{hasimoto1972}) is stationary under all but finitely many evolution equations of the NLS hierarchy.  
Grinevich \cite{grinevich2001} proved that closed finite gap curves in $\mathbb{R}^3$ are dense in the set of all closed curves in $\mathbb{R}^3$. This paper is the second part of two papers. In the first part \cite{kleinkilian1}, we extended Grinevich's result to closed curves in the round 3-sphere \,$\bbS^3$\,, and also showed that in both \,$\bbR^3$\, and \,$\bbS^3$\,, a closed curve can be approximated by closed finite gap curves of the same total torsion. In the present paper we prove the corresponding result for closed curves in hyperbolic 3-space \,$\bbH^3$\,, and for the 2-dimensional space forms. Putting all these results together then yields the following

\textbf{Theorem.} The set of closed finite gap curves in any two- or three-dimensional space form is dense in the space of all closed curves in that space form. In the three-dimensional space forms, closed curves can be approximated by closed finite gap curves of the same total torsion.

The shape of a curve in three dimensions is governed by two functions: its curvature and torsion. If the curve is closed, then these functions are periodic with commensurate periods. But taking two periodic functions with commensurate periods is not sufficient to obtain a closed curve - there are additional extrinsic closing conditions. These additional closing conditions can be written in terms of a spectral problem for a $2 \times 2$ matrix differential operator associated to the NLS hierarchy, and it is in this setting that we present our results. This is a well studied approach in the theory of integrable flows of curves, and there is a large body of literature on the subject, see \cite{langer1999, caliniivey2001, calini-ivey2005, goldstein-petrich1991} and the references therein.   

Like in \cite{kleinkilian1}, we use the spectral data and the associated perturbed Fourier coefficients derived from the closed curve via the NLS integrable system as the fundamental tool
for proving our results. The investigation of the spectral data carried out in \cite{kleinkilian1} (which is analogous to the investigation of the spectral data of the sinh-Gordon equation in \cite{klein-habil}) is fundamental also for the treatment of \,$\bbH^3$\, in the present paper. However, in comparison to the cases of \,$\bbR^3$\, and \,$\bbS^3$\, from \cite{kleinkilian1}, the case of \,$\bbH^3$\, is more complicated because the Sym points (i.e.~the spectral values at which the extended frame is evaluated in the Sym-Bobenko reconstruction formula)
of a closed curve in \,$\bbH^3$\, are not on the real line, and therefore the monodromy at these points is not necessarily semisimple. It follows that if we want to construct a
(finite gap) curve that satisfies the extrinsic closing condition for \,$\bbH^3$\,, it is not enough to make sure that the monodromy has the correct eigenvalues (\,$\pm 1$\,) at the Sym points. We have
to consider the case where the monodromy at the Sym points is not diagonalisable, and in this case we need to carry out an additional transformation to obtain a diagonalisable monodromy
with the same eigenvalues. We do this by applying a suitable Bianchi-B\"acklund transform, which we express as a simple factor dressing. For this technique to be successful, we need the eigenvalue function of the monodromy to attain the value \,$\pm 1$\, at the Sym points with second order. 

A brief outline of the paper is as follows: In section 1, we recall the $2 \times 2$ matrix model of hyperbolic 3-space $\bbH^3$, the complex curvature function of a curve, and its closing conditions in terms of the monodromy of its extended frame. In the second section we introduce spectral curves for periodic functions in $L^2([0,\,T],\,\bbC)$ and the eigenbundle of the monodromy. Section 3 deals with finite Mittag-Leffler distributions on the spectral curve with prescribed asymptotics. We use such distributions in section 4 to show that for every closed curve \,$\gamma$\, in \,$\bbH^3$\, there exists a sequence \,$(q_n)$\, of finite gap complex curvature functions which converge in \,$L^2$\, to the complex curvature of \,$\gamma$\,, such that the eigenvalue function of the monodromy of every \,$q_n$\, attains the value \,$\pm 1$\, at the Sym points with second order. In section 5, we use simple factor dressing in combination with the Baker-Akhiezer function to transform the \,$q_n$\, such that they fully satisfy the extrinsic closing condition for \,$\bbH^3$\,. In this way we prove that the set of closed finite gap curves in $\bbH^3$ is $W^{2,2}$-dense in the Sobolev space of all closed $W^{2,2}$ curves in $\bbH^3$. In the final section 6 we prove the result for curves in 2 dimensions and show that the set of closed finite gap curves in 2-dimensional space forms is $W^{2,2}$-dense in the Sobolev space of all closed $W^{2,2}$ curves in 2-dimensional space forms.

\textbf{Acknowledgements.} We thank Francis E.~Burstall and Martin U.~Schmidt for useful conversations. Sebastian~Klein is funded by the Deutsche Forschungsgemeinschaft, Grant 414903103.

%
%

\section{Extended frames of curves}
\label{Se:frames}

We identify hyperbolic 3-space \,$\bbH^3 \subset \bbR^{(3,1)}$\, as the symmetric space $\mathrm{SL}_2(\bbC) \slash \mathrm{SU}_2$ embedded in the real $4$-space of Hermitian symmetric matrices 
by the map \,$[g] \hookrightarrow g\,g^*$\,, where $g^*$ denotes the complex conjugate transpose of $g$.
The identity component of the isometry group of \,$\bbH^3$\, is isomorphic to the restricted Lorentz group \,$\mathrm{SO}^+(3,1)$\,,
whose double cover is \,$\mathrm{SL}_2(\bbC)$\,. \,$\mathrm{SL}_2(\bbC)$\, acts on \,$\bbH^3$\,  via the action 
\,$X \mapsto FXF^*$\,. We fix a basis of the Lie algebra \,$\mathfrak{sl}_2(\bbC)$\, by 
\begin{equation} \label{eq:epsilons}
  \varepsilon_- = \begin{pmatrix} 
  0 & 0 \\ -1 & 0 \end{pmatrix},\,
  \varepsilon_+ = \begin{pmatrix} 
  0 & 1 \\ 0 & 0 \end{pmatrix} \mbox{ and }
  \varepsilon = \begin{pmatrix}
  \mi & 0 \\ 0 & -\mi \end{pmatrix}\,.
\end{equation}  
The invariant Riemannian metric on \,$\bbH^3$\, induces an Ad-invariant inner product on \,$\mathfrak{sl}_2(\bbC)$\,, whose extension to a \,$\bbC$-bilinear map will be denoted by 
\,$\langle \cdot \, , \cdot \rangle$\,. Here we normalize the Riemannian metric in such a way that 
\,$\langle \varepsilon,\,\varepsilon \rangle = -\tfrac{1}{2} \tr\,\varepsilon^2 = 1$\, holds. We then moreover have
\begin{equation}\label{eq:commutators} \begin{split}
  &\langle \varepsilon_-,\,\varepsilon_- \rangle = 
  \langle \varepsilon_+,\,\varepsilon_+ \rangle = 0,\,
  \varepsilon_-^* = - \varepsilon_+,\,\varepsilon^* = -\varepsilon, \\
  &[\varepsilon_- ,\,\varepsilon ] = 2\mi \varepsilon_-,\,
  [\varepsilon,\, \varepsilon_+] = 2\mi \varepsilon_+ \mbox{ and }
  [\varepsilon_+,\, \varepsilon_-]= \mi\varepsilon. \end{split}
\end{equation}


For a curve \,$\gamma$\, in \,$\bbH^3$\, with geodesic curvature $\kappa$ and torsion $\tau$\,,
the complex-valued function $$q(t) = \kappa(t) e^{\mi \int_0^t \tau(s)\,ds}$$ is called the \emph{complex curvature} of $\gamma$, introduced by Hasimoto \cite{hasimoto1972}.
If \,$\gamma$\, is in \,$W^{k,2}$\, (i.e.~\,$k$-times differentiable in the Sobolev sense, where the final derivative is square-integrable) for \,$k\geq 2$\,,
then its complex curvature is in \,$W^{k-2,2}$\,. Note that the complex curvature is well-defined even when \,$\gamma$\, is only twice differentiable.
The complex curvature will be the main instrument with which we study curves in \,$\bbH^3$\,. The next lemma is a variant of the Frenet-Serret Theorem in our set-up, and is proven by direct computation. 
\begin{lemma} \label{th:sym}
Let \,$k \in \{2,3,\dotsc,\infty\}$\,, \,$T>0$\,, \,$q \in W^{k-2,2}([0,T],\C)$\, and
\begin{equation} \label{eq:alpha_q}
	\alpha^q = \tfrac{1}{2} \left( \lambda \,\varepsilon + q\,\varepsilon_+ + \bar{q}\,\varepsilon_- \right)
\end{equation}
Let $F = F(t,\,\lambda)$ be the unique solution of 
\begin{equation}\label{eq:frame1}
   \tfrac{d}{dt} F =  F \,\alpha^q \,,\qquad F(0,\,\lambda) = \mathbbm{1} \mbox{ for all } \lambda \in \mathbb{C}\,.
\end{equation}
Then \,$F(\,\cdot\,,\lambda) \in W^{k-1,2}([0,T],\bbC^{2\times 2})$\,. Moreover,
\begin{equation}
\label{eq:symbobenko-gamma}
\gamma(t) = F(t,\mi)\,F(t,-\mi)^{-1}
\end{equation}
is a unit-speed curve in $\bbH^3$ with complex curvature $q$,
with \,$\gamma \in W^{k,2}([0,T],\bbH^3)$\, and velocity
\begin{equation}
\label{eq:symbobenko-gamma'}
\gamma'(t) = \mi\,F(t,\mi)\,\varepsilon\,F(t,-\mi)^{-1}\,.
\end{equation}
\end{lemma}
%

The map $F: [0,T] \times \mathbb{C} \to \mathrm{SL}_2(\bbC)$ is called an {\emph{extended frame}} of the curve.
The spectral values \,$\lambda=\pm\mi$\, at which \,$F$\, is evaluated in the reconstruction formulas \eqref{eq:symbobenko-gamma} and \eqref{eq:symbobenko-gamma'} are called \emph{Sym points}.
\,$F$\, satisfies the {\emph{reality condition}} 
\begin{equation} \label{eq:reality}
	\overline{F(t,\,\bar\lambda)}^t = F^{-1}(t,\,\lambda) \quad \mbox{for all } \lambda \in \bbC \mbox{ and all } t \in \mathbb{R}\,.
\end{equation}
%
%
Now suppose \,$q\in L^2([0,T],\C)$\, is the complex curvature of a closed, unit speed curve \,$\gamma$\, in \,$\bbH^3$\, of length \,$T$\,. 
Then $q$ (extended naturally to \,$\bbR$\,) is in general not periodic, but only quasi-periodic, meaning that
\begin{equation} \label{eq:q-quasiperiodic}
q(t+T)=\exp(\mi\theta T)\, q(t) \quad \text{for all \,$t\in \R$\,,}
\end{equation}
with a unique number \,$\theta \in \R$\,. If \,$\gamma$\, is three times differentiable, so that its torsion function \,$\tau$\, is well-defined, then 
\,$\theta = \tfrac{1}{T} \int_0^T \tau \in \R$\, is the average torsion of the curve \,$\gamma$\,, so the total torsion of \,$\gamma$\, divided by its length.
We will take the liberty of calling the number \,$\theta T$\, the \emph{total torsion} of \,$\gamma$\, even where \,$\gamma$\, is only twice differentiable (and therefore
the torsion function of \,$\gamma$\, is not in general well-defined). 

In order to apply spectral theory to \,$\gamma$\,, we need a periodic potential. A periodic potential \,$\widetilde{q}$\, can be obtained from \,$q$\, by gauging \,$\alpha^q$\, with \,$g(t)$\,, see
\cite[p.~11]{grinevich-schmidt-sfb} and \cite[Section~1]{kleinkilian1}, where
$$ g(t) := \begin{pmatrix} \exp(\mi \theta t/2) & 0 \\ 0 & \exp(-\mi \theta t/2) \end{pmatrix} \; . $$
Indeed we have
$$ g.\alpha^q = g\,\alpha^q\,g^{-1} - \left( \tfrac{\mathrm{d}\ }{\mathrm{d}t} g \right)\,g^{-1} = \tfrac{1}{2} \left( \widetilde{\lambda} \,\varepsilon + \widetilde{q}\,\varepsilon_+ + \overline{\widetilde{q}}\,\varepsilon_- \right) = \alpha^{\widetilde{q}}$$
with
$$ \widetilde{q}(x) := \exp(-\mi \theta t)\,q(x) \quad\text{and}\quad \widetilde{\lambda} = \lambda+\theta \; . $$
The potential \,$\widetilde{q}$\, is periodic because of equation~\eqref{eq:q-quasiperiodic}, and the original curve \,$\gamma$\, can be reconstructed from the data \,$(\widetilde{q},\theta)$\,:
Let \,$F^{\widetilde{q}}$\, be the extended frame defined by \,$\widetilde{q}$\,, so the solution of the initial value problem
$$ \tfrac{d}{dt} F^{\widetilde{q}} =  F^{\widetilde{q}} \,\alpha^{\widetilde{q}} \,,\qquad F^{\widetilde{q}}(0,\,\widetilde{\lambda}) = \mathbbm{1} \mbox{ for all } \widetilde{\lambda} \in \mathbb{C}\,. $$
Due to Lemma~\ref{th:sym} we then have
\begin{equation}
\label{eq:gamma-H3}
\gamma(t) = F^{\widetilde{q}}(t,\,\mi+\theta)\, F^{\widetilde{q}}(t,\,-\mi+\theta)^{-1} \quad\text{and}\quad
\gamma'(t) = \mi\,F^{\widetilde{q}}(t,\,\mi+\theta)\, \varepsilon\, F^{\widetilde{q}}(t,\,-\mi+\theta)^{-1} \; . 
\end{equation}

From here on, we will denote by \,$q$\, the periodic potential that has been obtained from the complex curvature of \,$\gamma$\, by the above regauging. We will omit the tilde in the names of
\,$\widetilde{q}$\,, \,$\widetilde{\lambda}$\, and associated objects. Moreover we will henceforth usually omit the superscript \,${}^q$\, in \,$\alpha^q$\, and similar objects.

Even where the potential \,$q$\, is periodic, the extended frame \,$F$\, is generally not periodic. The extent of its non-periodicity is measured by the \emph{monodromy} 
\,$M(\lambda) := F(T,\lambda)$\,. Due to the periodicity of \,$q$\, we have
$$ F(t+T,\lambda)=M(\lambda)\, F(t,\lambda) \quad\text{for every \,$t\in \R$\,.} $$


%
\begin{remark} \label{th:monodromy-remarks}
The monodromy inherits the following properties from its extended frame.

{\rm{(i)}} From the ODE \eqref{eq:frame1} it follows that the map $\bbC \to \mathrm{SL}_2(\bbC),\,\lambda \mapsto M(\lambda)$ is analytic.

{\rm{(ii)}} From the reality condition \eqref{eq:reality} it follows that 
\begin{equation} \label{eq:mreality}
	\overline{M(\overline{\lambda})}^t = M(\lambda)^{-1} \quad \mbox{ for all } \lambda \in \bbC\,. 
\end{equation}
In particular $M(\lambda) \in \mathrm{SU}_2$ for all $\lambda \in \bbR$.

{\rm{(iii)}} The trace $\Delta: \lambda \mapsto \tr\, M(\lambda)$ is analytic and satisfies
\begin{equation} \label{eq:trace}
	 \overline{\Delta(\overline{\lambda})} = \Delta(\lambda)\,.
\end{equation}
{\rm{(iv)}} The two eigenvalues of $M$ are given by 
\begin{equation} \label{eq:eigenvalues}
	\mu^{\pm 1}  = \tfrac{1}{2} \bigl( \Delta \pm \sqrt{\Delta^2 - 4} \bigr)
\end{equation}
and satisfy $\Delta = \mu + \mu^{-1}$ and 
\begin{equation} \label{eq:mu-reality}
	\overline{\mu(\overline{\lambda})} = \mu(\lambda)^{-1}\,.
\end{equation}
{\rm{(v)}} The map $\lambda \mapsto \mu(\lambda)$ is branched at the odd ordered roots of $\Delta^2 - 4$, and there
\begin{equation} \label{eq:branch}
	\Delta = \pm 2 \Longleftrightarrow \mu = \pm 1\,.
\end{equation}
{\rm{(vi)}} The curve defined by equation \eqref{eq:gamma-H3} is $T$-periodic if and only if
\begin{equation} M(\mi+\theta) = M(-\mi+\theta) = \pm \mathbbm{1} \;. 
\end{equation}
\end{remark}

\section{The spectral curve}

In the sequel we denote by \,$L^2([0,T],\bbC)$\, the space of complex-valued square-integrable functions on \,$[0,T]$\,, and we will always regard such functions as being extended
periodically on the real line. 
We consider a periodic potential \,$q\in L^2([0,T],\bbC)$\,, let \,$F(x,\lambda)$\, be the corresponding extended frame and \,$M(\lambda)$\, be the associated monodromy. We also consider the
holomorphic function \,$\Delta := \tr M$\,, see Remark~\ref{th:monodromy-remarks}(iii).
The \emph{spectral curve} \,$\Sigma$\, of \,$q$\, is then defined by the characteristic equation of \,$M$\,:
$$ \Sigma = \bigr\{ \; (\lambda,\mu)\in \C^2 \;\bigr|\; \mu^2 - \Delta(\lambda)\,\mu+1=0 \;\bigr\} \; . $$
\,$\Sigma$\, is a complex curve, its singularities occur at those \,$\lambda \in \C$\, for which \,$\Delta^2-4$\, has a zero of order \,$\geq 2$\,. \,$\Sigma$\, is hyperelliptic above \,$\C$\,,
in other words,
$$ \pi: \Sigma \to \bbC, \; (\lambda,\mu) \mapsto \lambda $$
is a holomorphic, two-fold branched covering map; its branch points occur above those \,$\lambda\in \C$\, for which \,$\Delta^2-4$\, has a zero of odd order. 
The hyperelliptic involution of \,$\Sigma$\, is given by
$$ \sigma: \Sigma \to \Sigma, \; (\lambda,\mu) \mapsto (\lambda,\mu^{-1}) \; . $$
Remark~\ref{th:monodromy-remarks}(ii) shows that \,$\Sigma$\, is also equipped with the anti-holomorphic involution
$$ \eta: \Sigma\to\Sigma,\; (\lambda,\mu) \mapsto (\overline{\lambda},\overline{\mu}^{-1}) \; , $$
which commutes with \,$\sigma$\,. 
The eigenvector bundle of \,$M$\, is a holomorphic line bundle on a certain partial normalisation \,$\widetilde{\Sigma}$\, of \,$\Sigma$\,, which we call the \emph{eigenline curve} of \,$M$\,.
It is characterised by the property that the
generalised divisor (in the sense of Hartshorne \cite[\S 1]{Hartshorne}) which describes the eigenvector bundle is locally free on the branched one-fold covering \,$\widetilde{\Sigma}$\,
of \,$\Sigma$\,, see \cite[Section~4]{klss2016} (where \,$\widetilde{\Sigma}$\, is called the ``middleding'' of the holomorphic matrix \,$M$\,)
and also compare \cite[Section~3]{klein-habil}. The transpose matrix \,$M^t$\, has the same eigenline curve as \,$M$\,, and therefore the eigenvector
bundle of \,$M^t$\, also is a holomorphic line bundle on \,$\widetilde{\Sigma}$\,. Thus there exist non-zero holomorphic sections \,$v,w: \widetilde{\Sigma}\to\C^2$\, of these eigenline bundles so that
$$ M v = \mu\,v \quad\text{respectively}\quad M^t\,w = \mu\,w \; . $$
In terms of these sections, the projection operator onto the eigenline bundle of \,$M$\, is given by the meromorphic operator map on \,$\widetilde{\Sigma}$\,
$$ P := \frac{v \, w^t}{w^t \, v} \; ,$$
note that \,$P$\, does not change when \,$v$\, and \,$w$\, are scaled by non-zero holomorphic functions. In the sequel we will
also consider \,$u = v \circ \sigma^{-1}$\,.

\begin{proposition}
  \label{P:spectral:P}
  \begin{itemize}
  \item[(i)] \,$Pv=v$\,, \,$Pu=0$\,, \,$P^2=P$\,.    
  \item[(ii)] \,$(v,u)$\, is a basis of \,$\C^2$\, at all those \,$(\lambda,\mu)\in\widetilde{\Sigma}$\, with \,$\mu\neq \pm 1$\,.
  \item[(iii)] \,$\sum P = \one$\, and \,$\sum \mu P = M(\lambda)$\,, where \,$\sum$\, denotes the sum over the two sheets of \,$\widetilde{\Sigma}$\,.
  \item[(iv)] For any linear operator \,$A$\, on \,$\C^2$\, we have \,$\tr(P\, A) = \tfrac{w^t \, Av}{w^t \, v}$\,. In particular, \,$\tr(P\, M)=\mu$\,.    
  \end{itemize}
\end{proposition}

\begin{proof}
It suffices to check the statements for those points \,$(\lambda,\mu)\in\widetilde{\Sigma}$\, with \,$\mu\neq \pm 1$\, and \,$w^t\, v \neq 0$\,. The equation \,$Pv=v$\, then follows immediately from the definition of \,$P$\,,
  and if we write \,$M=\left( \begin{smallmatrix} a & b \\ c & d \end{smallmatrix} \right)$\,,
  we have \,$v=(b,\mu-a)$\,, \,$w=(c,\mu-a)$\, and \,$u=(b,\mu^{-1}-a)$\,. We have the equation
  $$ (\mu-a)(\mu-d) = bc \;, $$
  whence
\begin{equation*} \begin{split}
w^t\, u &= (c,\mu-a)\, (b,\mu^{-1}-a) = c\,b+(\mu-a)(\mu^{-1}-a) \\ &= cb+(\mu-a)((\Delta-\mu)-(\Delta-d)) = c\,b-(\mu-a)(\mu-d)=0 \;,
\end{split}
\end{equation*}
and thus \,$Pu=0$\, follows. Because of the equations \,$Pv=v$\, and \,$Pu=0$\,, the vectors \,$(v,u)$\, must be linear independent, and thus a basis of \,$\C^2$\,. This fact also implies \,$P^2=P$\,.
  This completes the proof of (i) and (ii). (iii) follows from (i), (ii) and the fact that the function \,$\mu$\, enumerates the eigenvectors of \,$M(\lambda)$\,. Finally (iv) was shown in
  \cite[Lemma~6.1]{kleinkilian1}.
\end{proof}

Because of \,$\det M =1$\,, we have \,$M^{-1} = J\,M^t\,J^{-1}$\, with \,$J := \varepsilon_+ +\varepsilon_-$\,; it therefore follows from Equation~\eqref{eq:mreality} that
$$ 
	\overline{M(\overline{\lambda})} = J\,M(\lambda)\,J^{-1} $$
holds. With the anti-holomorphic involution \,$\rho := \sigma \circ \eta = \eta\circ \sigma : (\lambda,\mu) \mapsto (\overline{\lambda},\overline{\mu})$\, we therefore have
\,$\rho^* \overline{v} = Jv$\, and \,$\rho^* \overline{w} = Jw$\,, and thus
\begin{equation}
\label{eq:spectral:rho-P}
\rho^* \overline{P} = J\,P\,J^{-1} \; .
\end{equation}

For \,$x\in \R$\, we define \,$\tau_xq(t)=q(t+x)$\,. Then we have \,$F^{\tau_x q}(t,\lambda)=F^q(x,\lambda)^{-1}\,F^q(t+x,\lambda)$\,,
and therefore the monodromy of \,$\tau_x q$\, is given by
\begin{equation} \label{eq:spectral:M-trans}
  M^{\tau_x q}(\lambda) = F^{\tau_x q}(T,\lambda) = F^q(x,\lambda)^{-1}\,F^q(T+x,\lambda) = F^q(x,\lambda)^{-1}\,M^q(\lambda)\,F^q(x,\lambda) \; .
\end{equation}
This equation shows in particular that \,$M^{\tau_x q}$\, is the solution of the differential equation
$$ \frac{\mathrm{d}\ }{\mathrm{d}x} M^{\tau_x q} = [M^{\tau_x q},\,\alpha^q] \; . $$
Equation~\eqref{eq:spectral:M-trans} also shows that \,$F^q(x,\lambda)^{-1}\,v$\, and~\,$F^q(x,\lambda)^t\,w$\, is an eigenvector of \,$M^{\tau_x q}(\lambda)$\, respectively of \,$M^{\tau_x q}(\lambda)^t$\, for the eigenvalue \,$\mu$\,. It follows that 
\begin{equation}
  \label{eq:spectral:P-trans}
  P^{\tau_x q} = F^q(x,\lambda)^{-1}\, P^q \, F^q(x,\lambda) \; .
\end{equation}
A complex curvature function \,$q$\, is finite gap if and only if the eigenline curve \,$\widetilde{\Sigma}$\, has finite arithmetic genus. Therefore \,$\widetilde{\Sigma}$\, can in this case be compactified,
obtaining a hyperelliptic curve above \,$\CPone$\,, which
we again denote by \,$\widetilde{\Sigma}$\,. For \,$\lambda\to\infty$\,, \,$\widetilde{\Sigma}$\, is approximated by the corresponding curve for the vacuum \,$q=0$\,, which shows that
\,$\infty$\, is not a branching point of the compactification \,$\widetilde{\Sigma}$\,, in other words there are two points \,$\infty_+, \infty_- \in \widetilde{\Sigma}$\, that are above
\,$\lambda=\infty \in \CPone$\,. 

Even in the general case for \,$q$\,, the asymptotic behavior of the monodromy \,$M$\, described in \cite[Theorem~2.1]{kleinkilian1} shows that \,$\Sigma$\, and \,$\widetilde{\Sigma}$\,
can be compactified as a 1-dimensional complex space in a certain sense by adding points above \,$\lambda=\infty$\,, analogously to the spectral curve \,$Y$\, in \cite[Definition~2.1]{schmidt1996}.
The comparison to the spectral curve of the vacuum (\,$q=0$\,) shows that these compactifications are not branched above \,$\lambda=\infty$\,. Thus they have two points above \,$\lambda=\infty$\,,
which we will denote by \,$\infty_+$\, and \,$\infty_-$\,. 

\section{Mittag-Leffler distributions on the spectral curve}

In the present section we will study Mittag-Leffler distributions on the eigenline curve \,$\widetilde{\Sigma}$\, of the spectral curve \,$\Sigma$\,.
For the concept of a Mittag-Leffler distribution, see for example \cite[\S 18.1]{forster}.
In particular we will obtain a criterion when a Mittag-Leffler distribution on \,$\widetilde{\Sigma}$\,
can be solved by a meromorphic function that tends to zero of order \,$O(\lambda^{-1})$\, for \,$\lambda\to\infty$\, (Proposition~\ref{P:ML:solvable}). The construction described here follows the strategy
employed by \textsc{Schmidt} in \cite[Chapter~7]{schmidt1996} and \cite[Section~3.1]{schmidtwillmore} for other types of integrable systems.
Like the second reference, and unlike the first reference, we will only consider finite Mittag-Leffler
distributions, to avoid certain topological difficulties. Our application of these results in Section~\ref{Se:closing} will concern a Mittag-Leffler distribution whose support has at most two points. However, most results of the present section can be transferred to infinite Mittag-Leffler distributions if suitable convergence conditions are introduced, like it was done
in \cite[Chapter~7]{schmidt1996}.

For the sake of simplicity of notation, we will \emph{denote the eigenline curve of \,$M$\, by \,$\Sigma$\, instead of \,$\widetilde{\Sigma}$\,} in this section (only).
We denote the sheaf of holomorphic functions on the eigenline curve \,$\Sigma$\, by \,$\calO$\,, and for any divisor \,$D$\, on \,$\Sigma$\,, the sheaf of meromorphic functions \,$f$\, on \,$\Sigma$\,
with \,$(f) \geq -D$\, by \,$\calO_D$\,. Moreover, we denote by \,$\Omega$\, the canonical bundle of \,$\Sigma$\,, i.e.~the sheaf of holomorphic 1-forms on \,$\Sigma$\,.

Let \,$D$\, be a divisor on \,$\Sigma$\, and \,$\mathfrak{U}=(U_i)_{i\in I}$\, an open covering of \,$\Sigma$\,. Then a \emph{Mittag-Leffler distribution} on \,$\Sigma$\, for the divisor \,$D$\, is a
cochain \,$\mu = (h_i) \in C^0(\mathfrak{U},\calO_D)$\, so that the differences \,$h_i-h_j$ are holomorphic on \,$U_i\cap U_j$\,, i.e.~so that \,$\delta \mu = Z^1(\mathfrak{U},\calO)$\, holds.
The cohomology class induced by \,$\mu$\, is \,$[\delta \mu] \in H^1(\Sigma,\calO)$\,. In the sequel, we identify Mittag-Leffler distributions \,$\mu$\, with cohomology classes
\,$h=[\delta \mu] \in H^1(\Sigma,\calO)$\,.
A \emph{solution} of \,$\mu$\, resp.~of \,$h$\, is a global meromorphic function \,$f \in C^0(\Sigma,\calO_D)$\, which has the same principal parts as \,$\mu$\,, i.e.~so that \,$f|U_i-h_i$\, is holomorphic for
every \,$i\in I$\,.

We say that the Mittag-Leffler distribution \,$h$\, is \emph{finite} if the support of the corresponding divisor \,$D$\, is finite,
and denote the space of cohomology classes corresponding to finite Mittag-Leffler distributions by \,$H^1_{fin}(\Sigma,\calO)$\,. In this section, we will focus on finite
Mittag-Leffler distributions.  

For \,$h\in H^1_{fin}(\Sigma,\calO)$\, and \,$\omega \in H^0(\Sigma,\Omega)$\, we define
$$ \Res(h,\omega) = -\sum \Res(h\,\omega) \;, $$
where the sum runs over the finitely many poles of \,$f$\, on \,$\Sigma$\,. This is equal to the sum of the residues of \,$h\omega$\, at \,$\infty_\pm$\,. In this way, we obtain a
non-degenerate pairing
$$ \Res: H^1_{fin}(\Sigma,\calO) \times H^0(\Sigma,\Omega) \to \bbC \; . $$

We now consider the subspace \,$H^1_{fin,\rho}(\Sigma,\calO)$\, of those \,$h\in H^1_{fin}(\Sigma,\calO)$\,
which are compatible with the anti-holomorphic involution \,$\rho: \Sigma\to\Sigma,\;(\lambda,\mu)\mapsto(\overline{\lambda},\overline{\mu})$\, in the sense that
\,$\rho_*\overline{h}=h$\, holds. Likewise we denote by \,$H^0_\rho(\Sigma,\Omega)$\, the subspace of those \,$\omega \in H^0(\Sigma,\Omega)$\, which satisfy \,$\rho^*\overline{\omega}=\omega$\,. 
Note that the pairing \,$\Res(h,\omega)$\, is real-valued for \,$h\in H^1_{fin,\rho}(\Sigma,\calO)$\, and \,$\omega \in H^0_\rho(\Sigma,\Omega)$\,. 

On the other hand, the real symplectic map 
$$ \Omega: L^2([0,T],\bbC) \times L^2([0,T],\bbC) \to \R,\; (\delta_1 q, \delta_2 q) \mapsto \frac12 \int_0^T \mathrm{Im}\bigr(\delta_1q(t)\, \overline{\delta_2q(t)}\bigr)\,\mathrm{d}t $$
on the Banach space \,$L^2([0,T])$\, of infinitesimal variations of the periodic potential \,$q$\, is non-degenerate, and can be seen as a non-degenerate real pairing of \,$L^2([0,T],\bbC)$\, with itself.

The main objective of this section is to show that these two pairings are in a certain sense conjugate to each other: Below, we will construct an \,$\R$-linear map
$$ \Phi: H^1_{fin,\rho}(\Sigma,\calO) \to L^2([0,T],\bbC), \; h \mapsto \delta_h q $$
so that with the \,$\R$-linear map 
$$ \Psi: L^2([0,T],\bbC) \to H^0_\rho(\Sigma,\Omega),\; \delta q \mapsto \frac{1}{\mu} \left(\frac{\partial \mu}{\partial q}\,\delta q\right)\mathrm{d}\lambda $$
the following diagram commutes:
\begin{center}
\begin{tikzpicture}
  \begin{scope}[scale=1.0]
    \draw (0,1.5) node (A) {$H^1_{fin,\rho}(\Sigma,\calO)$};
    \draw (6,1.5) node (B) {$L^2([0,T],\bbC)$};
    \draw (0,0) node (C) {$H^0_\rho(\Sigma,\Omega)$};
    \draw (6,0) node (D) {$L^2([0,T],\bbC)\;.$};
    \draw (0,0.75) node (E) {\LARGE $\times$};
    \draw (6,0.75) node (F) {\LARGE $\times$};
    \draw (3,0.75) node (G) {$\bbR$};
    \path[->,font=\scriptsize] (A) edge node[above] {$\Phi$} (B);
    \path[->,font=\scriptsize] (D) edge node[below] {$\Psi$} (C);
    \path[->,font=\scriptsize] (E) edge node[above] {$\Res$} (G);
    \path[->,font=\scriptsize] (F) edge node[above] {$\Omega$} (G);
  \end{scope}
\end{tikzpicture}
\end{center}
In other words, we will have
$$ \Res(h,\Psi(\delta q)) = \Omega(\Phi(h),\delta q) \quad\text{for any \,$h \in H^1_{fin,\rho}(\Sigma,\calO)$\, and \,$\delta q \in L^2([0,T],\bbC)$\,.} $$

\begin{proposition} \label{P:ML:Phi}

\rm{(i)} Let \,$h\in H^1_{fin,\rho}(\Sigma,\calO)$\, be given. There exist \,$W^{1,2}$-differentiable maps \,$a_+$\, into the \,$(2\times 2)$-matrix-valued entire maps in \,$\lambda$\,, and
    \,$a_-$\, into the \,$(2\times 2)$-matrix-valued meromorphic functions in \,$\lambda \in \bbC$\, with poles at most at the points below the support of \,$h$\, and which go to zero for \,$\lambda\to\infty$\,
    so that
    $$ a_+ + a_- = a:= \sum\, h\, P^{\tau_x q} $$
    holds; here \,$\sum$\, denotes the sum over the two sheets of the complex curve \,$\Sigma$\,. The matrix-valued function
    \begin{equation}
      \label{eq:ML:Phi:deltaalpha}
      \delta \alpha := \frac{\mathrm{d}\ }{\mathrm{d}x} a_- + [\alpha,\,a_-] = -\tfrac12 [\varepsilon,\,\Res (a\,\mathrm{d}\lambda)]
    \end{equation}
    is of the form
    \,$\tfrac12\, \left( \begin{smallmatrix} 0 & \delta q \\ - \overline{\delta q} & 0 \end{smallmatrix} \right)$\, with some \,$\delta q \in L^2([0,T],\bbC)$\,, in particular \,$\delta \alpha$\, is independent of \,$\lambda$\,.
    We define the \,$\bbR$-linear map \,$\Phi: H^1_{fin,\rho}(\Sigma) \to L^2([0,T],\bbC)$\, by \,$h \mapsto \delta q$\,.
  
\rm{(ii)}
    For every \,$h\in H^1_{fin,\rho}(\Sigma,\calO)$\, we have \,$\tfrac{\partial \mu}{\partial q}\, \Phi(h)=0$\, and \,$\tfrac{\partial v}{\partial q}\, \Phi(h)=-a_-(0)\,v$\,,
    with \,$a_-$\, as in (i).
    
\rm{(iii)}    If some \,$h \in H^1_{fin,\rho}(\Sigma,\calO)$\, satisfies \,$\Phi(h)=0$\,, then the Mittag-Leffler distribution \,$h$\, is solvable by means of a meromorphic function \,$f$\, on \,$\Sigma$\, which tends to zero of order \,$O(\lambda^{-1})$\, near \,$\infty_\pm$\,. 
\end{proposition}
\begin{proof} \rm{(i)} Let \,$h \in H^1_{fin,\rho}(\Sigma,\calO)$\, be given. Because \,$\Sigma$\, is non-compact, every Mittag-Leffler distribution on \,$\Sigma$\, is solvable; in other words, we can regard
  \,$h$\, as a meromorphic function on \,$\Sigma$\, which has poles as prescribed by the Mittag-Leffler distribution. Of course, this meromorphic function is not unique, and in general,
  it will not be bounded near \,$\infty_\pm$\,.

  Let \,$a(x) := \sum h\,P^{\tau_x q}$\,, where the sum runs over the two sheets of \,$\Sigma$\,. The map \,$a$\, is obviously invariant under \,$\sigma$\, and therefore induces a \,$(2\times 2)$-matrix-valued,
  meromorphic function on \,$\bbC\ni\lambda$\,, which has poles at the values of \,$\lambda$\, below the points in the support of \,$h$\,, and possibly also at \,$\lambda=\infty$\,.
  Because of Equation~\eqref{eq:spectral:P-trans} we have
  $$ a(x) = F(x,\lambda)^{-1}\,a(0)\,F(x,\lambda) \; , $$
  and differentiating with respect to \,$x$\, shows that \,$a$\, solves the Lax equation
  \begin{equation}
    \label{eq:ML:Phi:ddx-a}
    \tfrac{\mathrm{d}a}{\mathrm{d}x} =[a,\,\alpha]\; .
  \end{equation}
  Let \,$a_L$\, be the principal part of the Laurent series of \,$a$\, at \,$\lambda=\infty$\,, and let \,$a_c$\, be the finite value of \,$a-a_L$\, at \,$\lambda=\infty$\,. 
  Then \,$a_+ := a_L + a_c$\, is an entire, \,$(2\times 2)$-matrix-valued
  function, and \,$a_-:= a - a_+$\, is a meromorphic function on \,$\bbC \ni \lambda$\, which tends to zero of order \,$O(\lambda^{-1})$\, for \,$\lambda\to\infty$\,, and we have \,$a = a_+ + a_-$\,. Because of equation~\eqref{eq:ML:Phi:ddx-a}
  we have
  \begin{equation}
    \label{eq:ML:r1:deltaalpha-def}
    \delta \alpha := \tfrac{\mathrm{d}\,a_- }{\mathrm{d}x}  + [\alpha,\,a_-]= - \left( \tfrac{\mathrm{d}\,a_+ }{\mathrm{d}x} + [\alpha,\,a_+] \right)  \; .
  \end{equation}
  The second representation of \,$\delta \alpha$\, shows that \,$\delta \alpha$\, is a \,$(2\times 2)$-matrix-valued, entire function in \,$\lambda\in \bbC$\,. On the other hand, the first representation shows that \,$\delta\alpha$\, is bounded near \,$\lambda=\infty$\,. Therefore \,$\delta \alpha$\, is constant in \,$\lambda$\,. 

For \,$\lambda\to\infty$\,, we have \,$a_-\to 0$\, and \,$\tfrac{\mathrm{d}\,a_- }{\mathrm{d}x} \to 0$\,. Moreover \,$\lambda\,a_- \to -\Res (a\,\mathrm{d}\lambda)$\,,
  which is a finite \,$(2\times 2)$-matrix. Because \,$\delta \alpha$\, does not depend on \,$\lambda\in \C$\,, we therefore have
  \begin{align*}
    \delta \alpha & = \lim_{\lambda\to\infty} \left( \tfrac{\mathrm{d}\,a_- }{\mathrm{d}x}  + [\alpha,\,a_-] \right) \overset{\eqref{eq:alpha_q}}{=} \tfrac12 \lim_{\lambda\to\infty} [\lambda \varepsilon + q\varepsilon_+ + \overline{q}\varepsilon_-,\,a_-] \\
    & = \tfrac12 \lim_{\lambda\to\infty} [\varepsilon,\, \lambda\,a_-] = -\tfrac12 [\varepsilon,\Res (a\,\mathrm{d}\lambda)] \; .
  \end{align*}
This shows that the matrix \,$\delta \alpha$\, is off-diagonal. 
  
Because of the hypothesis \,$\rho^*\overline{h}=h$\, and equation~\eqref{eq:spectral:rho-P} we have \,$\overline{a_{\overline{\lambda}}} = J\,a_\lambda\,J^{-1}$\, and therefore
  also \,$\overline{a_{\pm,\overline{\lambda}}} = J\,a_{\pm,\lambda}\,J^{-1}$\,. Because
  \,$\delta \alpha$\, does not depend on \,$\lambda$\,, we have by equation~\eqref{eq:ML:r1:deltaalpha-def}
  \begin{equation*} 
    \overline{\delta \alpha}
     = \overline{\delta \alpha}|_{\overline{\lambda}} = \tfrac{\mathrm{d}\,\overline{a_{-,\overline{\lambda}}} }{\mathrm{d}x}  + \overline{[\alpha_{\overline{\lambda}},\,a_{-,\overline{\lambda}}]} 
     = J\,\tfrac{\mathrm{d}\,a_{-,\lambda} }{\mathrm{d}x}\,J^{-1} + [J\,\alpha_\lambda\,J^{-1}\,,\, J\,a_{-,\lambda}\,J^{-1}] = J\,\delta \alpha|_\lambda\,J^{-1} = J\,\delta \alpha\,J^{-1} \; .
  \end{equation*}
  This shows that \,$\delta \alpha$\, is of the form \,$\tfrac12 \left( \begin{smallmatrix} 0 & \delta q \\ - \overline{\delta q} & 0 \end{smallmatrix} \right)$\, with some \,$\delta q \in L^2([0,T],\bbC)$\,. 
      
\rm{(ii)}  We consider the multi-valued map
  $$ \delta v(x) := -\mu^{x/T}\,a_-(x)\,F(x,\lambda)^{-1}\,v \quad\text{with}\quad \mu^{x/T} := \exp\bigr( \tfrac{x}{T}\,\ln(\mu)\bigr) \; . $$
Note that \,$\delta v := \delta v(0) = \delta v(T) = -a_-(0)\,v$\, holds. Using $\tfrac{\mathrm{d}\,F}{\mathrm{d}x} = F\,\alpha$, we have
  \begin{align*}
    \tfrac{\mathrm{d}\, \delta v}{\mathrm{d}x}
    & = \tfrac{\ln(\mu)}{T}\,\delta v - \mu^{x/T} \,\tfrac{\mathrm{d}\,a_-}{\mathrm{d}x}\,F^{-1}\,v + \mu^{x/T}\,a_-\,\alpha\,F^{-1}\,v 
     = \tfrac{\ln(\mu)}{T}\,\delta v - \mu^{x/T}\,\left( \tfrac{\mathrm{d}\,a_- }{\mathrm{d}x} - a_-\,\alpha \right) \, F^{-1}\,v \\
    & \overset{\eqref{eq:ML:r1:deltaalpha-def}}{=} \tfrac{\ln(\mu)}{T}\,\delta v - \mu^{x/T}\,\left( \delta\alpha - \alpha\,a_- \right) \, F^{-1}\,v 
     = \left( \tfrac{\ln(\mu)}{T}\one-\alpha \right) \delta v - \mu^{x/T}\,\delta\alpha\,F^{-1}\,v \; .
  \end{align*}
  This equation is an inhomogeneous linear differential equation for the function \,$\delta v$\,. The unique solution of this differential equation with the initial value \,$\delta v(0)=\delta v$\,
  is given by
  $$ \delta v(x) = \mu^{x/T}\,F(x,\lambda)^{-1}\, \left( -\int_0^x F(t,\lambda)\,\delta\alpha(t)\,F(t,\lambda)^{-1}\,v\,\mathrm{d}t + \delta v \right) \;. $$
  In particular, we have
  $$ \delta v = \delta v(T) = \mu\,M(\lambda)^{-1}\, \left( -\int_0^T F(t,\lambda)\,\delta\alpha(t)\,F(t,\lambda)^{-1}\,v\,\mathrm{d}t + \delta v \right) \;, $$
  whence
  \begin{equation}
    \label{eq:ML:Phi:magic}
    (M(\lambda)-\mu\one)\,\delta v = -\int_0^T F(t,\lambda)\,\delta\alpha(t)\,F(t,\lambda)^{-1}\,\mathrm{d}t \, \mu v
  \end{equation}
  follows. On the other hand, from the differential equation \eqref{eq:frame1} it follows that \,$G := \tfrac{\partial F}{\partial q}\, \delta q$\, solves the initial value problem
  \,$\tfrac{\mathrm{d}\ }{\mathrm{d}x} G = G\,\alpha + F\,\delta\alpha$\, with \,$G(0)=0$\,, and therefore we have
  $$ G(x,\lambda) = \int_0^x F(t,\lambda)\,\delta\alpha(t)\,F(t,\lambda)^{-1}\,\mathrm{d}t \, F(x,\lambda) \;, $$
  in particular
  $$ \frac{\partial M(\lambda)}{\partial q}\, \delta q = \int_0^T F(t,\lambda)\,\delta\alpha(t)\,F(t,\lambda)^{-1}\,\mathrm{d}t \, M(\lambda) \; . $$
  By applying this equation to Equation~\eqref{eq:ML:Phi:magic}, we obtain
  \begin{equation}
    \label{eq:ML:Phi:more-magic}
    (M(\lambda)-\mu\one)\,\delta v = -\left( \frac{\partial M(\lambda)}{\partial q}\, \delta q \right)\,v \; . 
  \end{equation}
  On the other hand, by differentiating the equation \,$(M(\lambda)-\mu\one)\,v=0$\, in the direction of \,$\delta q$\,, we obtain
  \begin{equation}
    \label{eq:ML:Phi:eigen-variation}
    (M(\lambda)-\mu\one)\,\left( \tfrac{\partial v}{\partial q}\, \delta q \right) = -\left( \tfrac{\partial M(\lambda)}{\partial q}\, \delta q \right)\,v + \left( \tfrac{\partial \mu}{\partial q}\, \delta q\right)\,v \; .      
  \end{equation}
  By comparing the symmetric and the anti-symmetric parts of Equations~\eqref{eq:ML:Phi:more-magic} and \eqref{eq:ML:Phi:eigen-variation}, we obtain \,$\tfrac{\partial \mu}{\partial q}\,\delta q = 0$\,
  and \,$\tfrac{\partial v}{\partial q}\, \delta q = \delta v = -a_-(0)\,v$\,. 

\rm{(iii)}  If \,$\delta q=\Phi(h)=0$\, holds, then we have \,$(M - \mu\one)\,\delta v=0$\, by Equation~\eqref{eq:ML:Phi:more-magic}, so that \,$\delta v$\, is an
  eigenvector of \,$M$\, for the eigenvalue \,$\mu$\,. Because the eigenspaces of \,$M$\, are generically one-dimensional, there exists a meromorphic function \,$f$\, on \,$\Sigma$\,
  with \,$\delta v = f\, v$\,. Because of the construction of \,$\delta v = -a_-(0)\,v$\,, the function \,$f$\, is a solution of the Mittag-Leffler distribution \,$h$\, which tends to zero of order \,$O(\lambda^{-1})$\,
  near \,$\infty_\pm$\,. 
\end{proof}

\begin{proposition} \label{P:ML:Psi}
\rm{(i)} If \,$\delta q \in L^2([0,T],\bbC)$\, then \,$\Psi(\delta q) := \tfrac{1}{\mu}\left( \tfrac{\partial \mu}{\partial q}\, \delta q \right) \mathrm{d}\lambda$\, is a regular 1-form on \,$\Sigma$\,
    which satisfies \,$\rho^* \overline{\Psi(\delta q)}=\Psi(\delta q)$\,. 
    We thus obtain an \,$\bbR$-linear map
    $$ \Psi: L^2([0,T],\bbC) \to H^0_\rho(\Sigma,\Omega) \; . $$
\rm{(ii)} For \,$\delta q \in L^2([0,T],\bbC)$\, we have
    \begin{equation}
      \label{eq:ML:Psi:intformula}
      \Psi(\delta q) = \left( \int_0^T \tr(P^{\tau_t q}\, \delta \alpha(t))\,\mathrm{d}t \right) \mathrm{d}\lambda \quad\text{with}\quad \delta\alpha := \frac12 \begin{pmatrix} 0 & \delta q \\ -\overline{\delta q} & 0 \end{pmatrix} \; .
    \end{equation}
\end{proposition}
\begin{proof} (i) The equation \,$\rho^* \overline{\Psi(\delta q)}=\Psi(\delta q)$\, follows directly from the definitions of \,$\Psi$\, and \,$\rho(\lambda,\mu)=(\overline{\lambda},\overline{\mu})$\,. 

(ii) It was shown in \cite[Lemma~6.3]{kleinkilian1} that
  $$ \frac{\partial \mu}{\partial q}\, \delta q = \frac{\mu}{w^t\, v} \int_0^T w(t)^t \, \delta\alpha(t)\, v(t)\,\mathrm{d}t $$
  holds, where \,$v(t) := F(t)^{-1}v$\, and \,$w(t) := F(t)^tw$\,. By the equality \,$w^t \, v = w(t)^t \, v(t)$\, and Proposition~\ref{P:spectral:P}(iv), we thus obtain
  $$  \frac{\partial \mu}{\partial q}\, \delta q
    = \mu \, \int_0^T \frac{w(t)^t \, \delta\alpha(t)\, v(t)}{w(t)^t \, v(t)}\,\mathrm{d}t 
    = \int_0^T \tr(\,P(\tau_t q)\, \delta \alpha(t) \,)\,\mathrm{d}t \; , $$
and equation~\eqref{eq:ML:Psi:intformula} follows.
\end{proof}

\begin{proposition}
  \label{P:ML:Res}  
  For \,$h \in H^1_{fin,\rho}(\Sigma,\calO)$\, and \,$\delta q \in L^2([0,T],\bbC)$\, we have
  $$ \Res(h,\Psi(\delta q)) = \Omega(\Phi(h),\delta q) \; . $$
\end{proposition}

\begin{proof}
  Let \,$h \in H^1_{fin,\rho}(\Sigma,\calO)$\, and \,$\delta q \in L^2([0,T],\bbC)$\, be given, and put
  \,$\delta \alpha := \tfrac12 \left( \begin{smallmatrix} 0 & \delta q \\ -\overline{\delta q} & 0 \end{smallmatrix} \right)$\,.
  We define \,$a(x) := \sum h\, P^{\tau_x q}$\, as in Proposition~\ref{P:ML:Phi}. 
  By Proposition~\ref{P:ML:Psi}(ii) we then have
  \begin{align*}
    \Res(h,\Psi(\delta q))
    &=\; \Res\left( \, h \,,\, \left( \int_0^T \tr\bigr( P^{\tau_t q}\,\delta\alpha(t) \bigr) \,\mathrm{d}t \right)\,\mathrm{d}\lambda \right) 
    = \int_0^T \tr\bigr( \Res_\Sigma(h\,P^{\tau_t q},\mathrm{d}\lambda)\, \delta\alpha(t)\bigr) \,\mathrm{d}t \\
    &=\;  \int_0^T \tr\bigr( \Res_{\bbC}(\sum h\,P^{\tau_t q}\,\mathrm{d}\lambda)\, \delta\alpha(t)\bigr) \,\mathrm{d}t 
    = \int_0^T \tr\bigr( \Res_{\bbC}(a(t)\,\mathrm{d}\lambda)\, \delta\alpha(t)\bigr) \,\mathrm{d}t \; . 
  \end{align*}
  Because \,$\delta\alpha$\, is an off-diagonal \,$(2\times 2)$-matrix, the value of the trace in the above expression depends among the entries of \,$\Res_{\bbC}(a\,\mathrm{d}\lambda)$\,
  only on the off-diagonal ones. The off-diagonal entries of \,$\Res_{\bbC}(a\,\mathrm{d}\lambda)$\, are equal to the off-diagonal entries of the matrix
  $$ -\frac14[\varepsilon,[\varepsilon,\,\Res_{\bbC}(a\,\mathrm{d}\lambda)]] = \frac12\left[\varepsilon,\,\frac12 \begin{pmatrix} 0 & \Phi(h) \\ -\overline{\Phi(h)} & 0 \end{pmatrix} \right] = \frac{\mi}{2} \, \begin{pmatrix} 0 & \Phi(h) \\ \overline{\Phi(h)} & 0 \end{pmatrix} \; , $$
  where the first equals sign follows from Equation~\eqref{eq:ML:Phi:deltaalpha}. Thus we obtain
  \begin{align*}
    \Res(h,\Psi(\delta q))
    &=\; \int_0^T \tr\left( \frac{\mi}{2} \begin{pmatrix} 0 & \Phi(h) \\ \overline{\Phi(h)} & 0 \end{pmatrix} \, \frac12 \begin{pmatrix} 0 & \delta q \\ -\overline{\delta q} & 0 \end{pmatrix} \right)\,\mathrm{d}t \\
    &=\; \frac{\mi}{4} \int_0^T \bigr( {-}\Phi(h)\,\overline{\delta q} + \overline{\Phi(h)}\,\delta q \bigr) \,\mathrm{d}t 
    = \frac12 \int_0^T \mathrm{Im}\bigr( \Phi(h)\,\overline{\delta q}\bigr) \,\mathrm{d}t = \Omega(\Phi(h),\delta q) \; . 
  \end{align*}
\end{proof}

\begin{proposition}
  \label{P:ML:solvable}
  Let \,$h \in H^1_{fin}(\Sigma,\calO)$\, be a finite Mittag-Leffler distribution. If
  \begin{equation}
    \label{eq:ML:solvable:hyp}
    \Res(h,\Psi(\delta q))=0 \quad\text{holds for every \,$\delta q \in L^2([0,T],\bbC)$\,,}
  \end{equation}
  then \,$h$\, is solvable by means of a meromorphic function \,$f$\, on \,$\Sigma$\, which tends to zero of order \,$O(\lambda^{-1})$\, near \,$\infty_\pm$\,. 
\end{proposition}

\begin{proof}
  Let us first suppose that \,$h\in H^1_{fin,\rho}(\Sigma,\calO)$\,, i.e.~\,$\rho_* \overline{h}=h$\, holds. Then the hypothesis \eqref{eq:ML:solvable:hyp} implies by Proposition~\ref{P:ML:Res}
  that \,$\Omega(\Phi(h),\delta q)=0$\, holds for every \,$\delta q \in L^2([0,T],\bbC)$\,. Because \,$\Omega$\, is non-degenerate, \,$\Phi(h)=0$\, follows. Therefore
  Proposition~\ref{P:ML:Phi}(iii) shows that the Mittag-Leffler distribution \,$h$\, can be solved by means of a meromorphic function \,$f$\, on \,$\Sigma$\, which tends to zero of order \,$O(\lambda^{-1})$\, near \,$\infty_\pm$\,.

  Now we consider an arbitrary \,$h\in H^1_{fin}(\Sigma,\calO)$\,. Then let us define
  $$ h_1 := \frac12\,(h+\rho_* \overline{h}) \quad\text{and}\quad h_2 := -\frac{\mi}{2}\,(h - \rho_* \overline{h}) \; . $$
  With this choice we have \,$h = h_1 + \mi\,h_2$\, and \,$\rho_* \overline{h_k} = h_k$\, for \,$k\in \{1,2\}$\,. Moreover, for any \,$\delta q \in L^2([0,T],\bbC)$\, we have due to
  the equation \,$\rho^* \overline{\Psi(\delta q)} = \Psi(\delta q)$\, (see Proposition~\ref{P:ML:Psi}(i)) and the hypothesis \eqref{eq:ML:solvable:hyp}
  \begin{align*}
    \Res\bigr(h_1,\Psi(\delta q)\bigr) & = \Res\bigr( \tfrac12(h+\rho_* \overline{h} ),\Psi(\delta q)\bigr) = \tfrac12\bigr( \Res(h,\Psi(\delta q)) + \Res(\rho_* \overline{h},\Psi(\delta q))\bigr) \\
                                                    & = \tfrac12\bigr( \Res(h,\Psi(\delta q)) + \overline{\Res(h,\rho^*\overline{\Psi(\delta q)})}\bigr)
                                                      = \tfrac12\bigr( \Res(h,\Psi(\delta q)) + \overline{\Res(h,\Psi(\delta q))}\bigr)\\
                                                    & = \mathrm{Re}\bigr( \Res(h,\Psi(\delta q)) \bigr) = 0
  \end{align*}
  and similarly
  $$ \Res\bigr(h_2,\Psi(\delta q)\bigr) = \mathrm{Im}\bigr( \Res(h,\Psi(\delta q)) \bigr) = 0 \; . $$
  The first part of the proof therefore shows that for \,$k\in \{1,2\}$\,, there exists a meromorphic function \,$f_k$\, on \,$\Sigma$\, which tends to zero of order \,$O(\lambda^{-1})$\, near \,$\infty_\pm$\, and solves the
  Mittag-Leffler distribution \,$h_k$\,. The meromorphic function 
  \,$f := f_1 + \mi\,f_2$\, on \,$\Sigma$\, then tends to zero of order \,$O(\lambda^{-1})$\, near \,$\infty_\pm$\, and solves the Mittag-Leffler distribution \,$h$\,. 
\end{proof}

\section{The closing condition for curves in $\bbH^3$}
\label{Se:closing}

\begin{lemma}
\label{lem:closing}
Let \,$q\in L^2([0,T],\bbC)$\, be periodic and \,$M$\, be the corresponding monodromy and \,$\Delta := \tr M$\,.
Also let \,$\theta\in\R$\,. Then the closing condition for \,$\bbH^3$\, in Remark~\ref{th:monodromy-remarks}(vi) for \,$\theta$\,
is equivalent to the following condition:
$$ \Delta(\mi+\theta)=\pm 2 \quad\text{and}\quad \text{\,$M(\mi+\theta)$\, is semisimple.} $$
If this condition holds, then we have \,$\Delta'(\mi+\theta)=0$\,. 
\end{lemma}

\begin{proof}
We have \,$M(-\mi+\theta)=(\overline{M(\mi+\theta)}^t)^{-1}$\, by Remark~\ref{th:monodromy-remarks}(ii),
and therefore the closing condition for \,$\bbH^3$\,, \,$M(\mi+\theta)=M(-\mi+\theta)=\pm\one$\,, is already implied by the simpler condition \,$M(\mi+\theta)=\pm \one$\,. The latter condition is in turn
equivalent to the condition that \,$\mu(\mi+\theta)=\pm 1$\, and \,$M(\mi+\theta)$\, is semisimple; note that \,$M(\mi+\theta)$\, is not automatically semisimple. The equation \,$\mu(\mi+\theta)=\pm 1$\,
is in turn equivalent to \,$\Delta(\mi+\theta) = \pm 2$\,.

Write \,$M=\left( \begin{smallmatrix} a & b \\ c & d \end{smallmatrix} \right)$\,. Because of \,$ad-bc=1$\,, we have
\begin{equation}
\label{eq:closing:Delta2-4-magic}
\Delta^2-4 = (a+d)^2-4(ad-bc) = (a-d)^2 + 4bc \; .
\end{equation}
If \,$M(\mi+\theta)=\pm\one$\, holds, then the holomorphic functions \,$a-d$\,, \,$b$\, and \,$c$\, are zero at \,$\mi+\theta$\,, and therefore
equation~\eqref{eq:closing:Delta2-4-magic} shows that \,$\Delta^2-4$\, has (at least) a second order zero at \,$\lambda=\mi+\theta$\,. Because of
\,$0=\left. \tfrac{\mathrm{d}\ }{\mathrm{d}\lambda} \right|_{\lambda=\mi+\theta}(\Delta^2-4) = 2\Delta(\mi+\theta)\,\Delta'(\mi+\theta)$\,, this implies \,$\Delta'(\mi+\theta)=0$\,. 
\end{proof}

\begin{lemma}
\label{lem:resrep}
Let \,$\lambda_*\in \C$\,. We define two Mittag-Leffler distributions \,$h_1,h_2 \in H^1_{fin}(\widetilde{\Sigma},\calO)$\,: Their support consists of the one or two points of \,$\widetilde{\Sigma}$\, above \,$\lambda_*$\,,
and at these points they are defined by the meromorphic germs
$$ h_1 \sim -\tfrac{\mu-\mu^{-1}}{2(\lambda-\lambda_*)} \quad\text{respectively}\quad h_2 \sim -\tfrac{\mu-\mu^{-1}}{2(\lambda-\lambda_*)^2} \; . $$
The \,$h_k$\, are anti-symmetric (meaning that \,$\sigma^* h_k = -h_k$\, holds), and for any \,$\delta q \in L^2([0,T],\bbC)$\, we have
\begin{equation}
\label{eq:resrep:res}
\Res(h_1,\Psi(\delta q)) = \tfrac{\partial \Delta(\lambda_*)}{\partial q} \, \delta q \quad\text{and}\quad
\Res(h_2,\Psi(\delta q)) = \tfrac{\partial \Delta'(\lambda_*)}{\partial q} \, \delta q \; ,
\end{equation}
where \,$\Psi: L^2([0,T],\bbC) \to H^0_\rho(\Sigma,\Omega)$\, is defined in Proposition~\ref{P:ML:Psi}.
\end{lemma} 
\begin{proof}
We first note that we have \,$\Delta = \mu + \mu^{-1}$\, and therefore
\begin{equation}
\label{eq:resrep:deltadelta}
\tfrac{\partial \Delta}{\partial q}\, \delta q = \tfrac{\mu-\mu^{-1}}{\mu} \left( \tfrac{\partial \mu}{\partial q}\, \delta q \right) \; .
\end{equation}
Hence we have
\begin{equation}
\label{eq:resrep:h1}
  \Res(h_1,\Psi(\delta q))
  = \Res_{\lambda_*} \left( \tfrac{\mu-\mu^{-1}}{2(\lambda-\lambda_*)} \, \tfrac{1}{\mu}\left( \tfrac{\partial \mu}{\partial q}\, \delta q \right) \,\mathrm{d}\lambda \right) 
  = \Res_{\lambda_*} \left( \left( \tfrac{\partial \Delta}{\partial q}\, \delta q \right) \,\tfrac{\mathrm{d}\lambda}{2(\lambda-\lambda_*)} \right) 
\end{equation}
and
\begin{equation}
  \label{eq:resrep:h2}
  \Res(h_2,\Psi(\delta q)) = \Res_{\lambda_*} \left( \left( \frac{\tfrac{\partial \Delta}{\partial q}\, \delta q}{\lambda-\lambda_*} \right) \,\frac{\mathrm{d}\lambda}{2\,(\lambda-\lambda_*)} \right) \; . 
\end{equation}
  
If \,$\Delta^2-4$\, is either non-zero at \,$\lambda_*$\, (then \,$\Sigma$\, has two regular non-branch points above \,$\lambda_*$\,) or has a zero of even order (then \,$\Sigma$\,
has one singular non-branch point above \,$\lambda_*$\,), then the eigenline curve \,$\widetilde{\Sigma}$\, has two points above \,$\lambda_*$\,, \,$\lambda-\lambda_*$\, is a coordinate
near these points, and we have \,$\Res(\tfrac{\mathrm{d}\lambda}{\lambda-\lambda_*})=2$\,. Therefore Equations~\eqref{eq:resrep:res} follow directly from Equations~\eqref{eq:resrep:h1}
and \eqref{eq:resrep:h2}.

If \,$\Delta^2-4$\, has a zero of odd order \,$2k-1$\, at \,$\lambda_*$\,, then \,$\Sigma$\, has a branch point at \,$\lambda_*$\,, which is regular for \,$k=1$\,, singular for
\,$k\geq 2$\,. In either case, the eigenline curve \,$\widetilde{\Sigma}$\, of \,$\Sigma$\, has only one point above \,$\lambda_*$\,, and \,$\sqrt{\lambda-\lambda_*}$\, is a coordinate
near this point. We note that \,$\tfrac{\mathrm{d}\lambda}{2(\lambda-\lambda_*)} = \tfrac{\mathrm{d}\sqrt{\lambda-\lambda_*}}{\sqrt{\lambda-\lambda_*}}$\, holds,
therefore Equations~\eqref{eq:resrep:res} again follow from Equations~\eqref{eq:resrep:h1} and \eqref{eq:resrep:h2}.
\end{proof}


\begin{lemma}
\label{lem:closing-linindep}
Suppose that the arithmetic genus of the eigenline curve \,$\widetilde{\Sigma}$\, is at least \,$4$\, (and possibly infinite), and let \,$\theta\in\R$\, be given. 
Then the four \,$\R$-linear forms \,$\eta_1,\dotsc,\eta_4: L^2([0,T],\bbC) \to \C$\, given by
$$
\eta_1(\delta q) = \frac{\partial \Delta(\mi+\theta)}{\partial q}\, \delta q \;,\quad \eta_2(\delta q) = \frac{\partial \Delta'(\mi+\theta)}{\partial q}\, \delta q \;,\quad
\eta_3(\delta q) = \overline{\eta_1(\delta q)} \quad\text{and}\quad \eta_4(\delta q) = \overline{\eta_2(\delta q)} $$
are linearly independent, i.e. there exist four variations \,$\delta_1 q,\dotsc,\delta_4 q \in L^2([0,T],\bbC)$\, of \,$q$\, so that the matrix \,$(\eta_j(\delta_k q))_{j,k=1,\dotsc,4}$\, has
maximal rank.
Moreover, the variations \,$\delta_k q$\, can be chosen so that only finitely many of their Fourier coefficients are non-zero. 
\end{lemma}

\begin{proof}
If variations \,$\delta_k q$\, as in the lemma exist at all, then they can be chosen with only finitely many non-zero Fourier coefficients, because the set of \,$L^2$-functions with
finitely many non-zero Fourier coefficients is dense in \,$L^2([0,T],\bbC)$\,, and the condition that the matrix \,$(\eta_j(\delta_k q))_{j,k=1,\dotsc,4}$\, has maximal rank is open.

Because of Remark~\ref{th:monodromy-remarks}(ii), \,$\overline{\Delta(\mi+\theta)} = \Delta(-\mi+\theta)$\, holds, and therefore we have for any \,$\delta q \in L^2([0,T],\bbC)$\,
$$ \eta_3(\delta q) = \frac{\partial \Delta(-\mi+\theta)}{\partial q}\, \delta q \;,\quad \eta_4(\delta q) = \frac{\partial \Delta'(-\mi+\theta)}{\partial q}\, \delta q \;. $$
Thus Lemma~\ref{lem:resrep} shows that there exist anti-symmetric Mittag-Leffler distributions \,$h_1,\dotsc,h_4 \in H^1_{fin}(\widetilde{\Sigma},\calO)$\, (with support at \,$\pm\mi+\theta$\,)
so that we have for any \,$\delta q \in L^2([0,T],\bbC)$\,
\begin{equation}
\label{eq:closing-linindep:etak-Res}
\eta_k(\delta q) = \Res(h_k, \Psi(\delta q)) \; .
\end{equation}


We now assume that the four linear forms \,$\eta_1,\dotsc,\eta_4$\, 
are linearly dependent, i.e.~that there exist constants \,$s_1,\dotsc,s_4 \in \C$\,, not all of them zero, so that
$$ s_1\,\eta_1 +  s_2\,\eta_2 +  s_3\,\eta_3 +  s_4\,\eta_4 =0 $$
holds. Equation~\eqref{eq:closing-linindep:etak-Res} then implies that we have
$$ \Res(h, \Psi(\delta q))=0 $$
for the anti-symmetric Mittag-Leffler distribution \,$h := s_1\,h_1 + s_2\,h_2 + s_3\,h_3 + s_4\,h_4 \in H^1_{fin}(\widetilde{\Sigma},\calO)$\, with support contained in \,$\{\mi+\theta,-\mi+\theta\}$\,
and every \,$\delta q \in L^2([0,T],\bbC)$\,. Proposition~\ref{P:ML:solvable} shows that \,$h$\, is solvable by means of a meromorphic function \,$f$\, on \,$\widetilde{\Sigma}$\, which
tends to zero near \,$\infty_\pm$\,. Because \,$h$\, is anti-symmetric, the function \,$f + f\circ \sigma$\, is holomorphic in \,$\lambda\in \C$\, and tends to zero near \,$\lambda=\infty$\,,
hence vanishes. Thus the solving function \,$f$\, is anti-symmetric, i.e.~\,$f \circ \sigma = -f$\, holds. 


We consider the hyperelliptic complex curve defined by \,$f$\,:
$$ Y = \bigr\{(\lambda,y)\in \C \times \C \mid y^2 = f^2(\lambda) \bigr\} \; . $$
Here \,$f^2$\, is a meromorphic function in \,$\lambda\in\C$\, because \,$f$\, is anti-symmetric with respect to \,$\sigma$\,. The Mittag-Leffler distribution \,$h$\,, and therefore also the
meromorphic function \,$f$\, has degree \,$d \leq 8$\,. Then \,$f^2$\, has degree \,$2d$\, as a function on \,$\widetilde{\Sigma}$\,, or degree \,$d$\, as a function in \,$\lambda\in\C$\,.
The arithmetic genus \,$g$\, of \,$Y$\, is given by \,$2g+2=d$\, \cite[III.7.4]{farkaskra}, and therefore the arithmetic genus of \,$Y$\, is at most \,$3$\,.

Because both \,$\widetilde{\Sigma}$\, and \,$Y$\, are hyperelliptic above \,$\C$\,, the meromorphic map
$$ \widetilde{\Sigma}\to Y, \; (\lambda,\mu) \mapsto (\lambda,f(\lambda,\mu)) $$
is a one-fold covering map. Therefore the arithmetic genus of \,$\widetilde{\Sigma}$\, is \,$\leq g \leq 3$\,, which is a contradiction to the hypothesis of the lemma.
\end{proof}

\begin{remark}
The statement in the treatment of \,$\R^3$\, and \,$\bbS^3$\, in \cite{kleinkilian1} that is analogous to Lemma~\ref{lem:closing-linindep} here is \cite[Lemma~6.4]{kleinkilian1}.
The same methods that are used here could also have been used for the proof of \cite[Lemma~6.4]{kleinkilian1}.
\end{remark}

In \cite[Section~3]{kleinkilian1} we introduced perturbed Fourier coefficients for the potential \,$q\in L^2([0,T],\bbC)$\,: In \cite[Lemma~3.1]{kleinkilian1} it was shown that the
points \,$\lambda\in \bbC$\, where the two diagonal entries of the monodromy \,$M$\, are equal can be enumerated by a sequence \,$(\lambda_k)_{k\in \bbZ}$\, in such a way that
\,$\lambda_k-\lambda_{k,0} \in \ell^2$\, holds, where \,$\lambda_{k,0} := \tfrac{2\pi}{T}\,k$\,. Then \,$z_k := 2(-1)^k\,b(\lambda_k) \in \ell^2$\, holds by \cite[Proposition~3.3]{kleinkilian1},
where \,$b$\, denotes the upper-right entry of \,$M$\,. We call the sequence \,$(z_k)_{k\in\bbZ}$\, the  \emph{perturbed Fourier coefficients} of \,$q$\,. The potential \,$q$\, is of finite gap if and only if all but finitely many of the perturbed Fourier coefficients are zero (\cite[Proposition~3.4]{kleinkilian1}). 

For \,$q\in L^2([0,T],\bbC)$\, and \,$N\in \bbN$\, we define the Banach space
$$ L^2([0,T],\bbC)_{q,N} := \{\;q_1\in L^2([0,T],\bbC)\;|\; \widehat{q_1}(k)=\widehat{q}(k) \text{ for all \,$|k|\leq N$\,}\;\} \; . $$
Here we denote for any \,$f \in L^2([0,T],\bbC)$\, and \,$k\in \bbZ$\, the (usual) \,$k$-th Fourier coefficient by
$$ \widehat{f}(k) := \int_0^T f(t)\,\exp(\tfrac{2\pi \mi}{T}\,k\,t) \,\mathrm{d}t \; . $$
In \cite[Proposition~5.2]{kleinkilian1} it was shown that for sufficiently large \,$N \in \bbN$\,, the map
$$ \Phi_N : L^2([0,T],\bbC)_{q,N} \to \ell^2(|k|>N),\; q_1 \mapsto (z_k(q_1))_{|k|>N} $$
is a local diffeomorphism near \,$q$\,. The following theorem shows how to integrate this result with the first ``half'' of the closing condition for \,$\bbH^3$\, in Lemma~\ref{lem:closing},
i.e.~with the additional conditions that \,$\Delta(\mi+\theta)=\pm 2$\, and \,$\Delta'(\mi+\theta)=0$\, holds. It is analogous to \cite[Theorem~6.5]{kleinkilian1} in the treatment
of \,$\bbR^3$\, and \,$\bbS^3$\,. 

\begin{theorem}
\label{thm:Psi-diffeo}
Suppose that \,$q$\, satisfies the closing condition for \,$\bbH^3$\, of Lemma~\ref{lem:closing} for some \,$\theta\in\R$\,. Also suppose that the arithmetic genus of \,$\widetilde{\Sigma}$\, is at least \,$4$\,.
Then there exist \,$N\in\bbN$\, and \,$f_1,\dotsc,f_4 \in L^2([0,T],\bbC)$\, so that the map
\begin{align*}
  \Psi_{N,f}: L^2([0,T],\bbC)_{q,N} + \bbR\,f_1 + \bbR\,f_2+ \bbR\,f_3+ \bbR\,f_4 & \to \ell^2(|k|> N) \times \bbC \times \bbC, \\
  q_1 & \mapsto \bigr(\,(z_k(q_1))_{|k|> N}\,,\,\Delta(\mi+\theta)\,,\,\Delta'(\mi+\theta)\,\bigr)
\end{align*}
is a local diffeomorphism near \,$q$\,.
\end{theorem}

\begin{proof}
We need to show that the derivative of \,$\Psi_{N,f}$\, at \,$q$\, is an isomorphism of real Banach spaces.
By Lemma~\ref{lem:closing-linindep} there exist four variations \,$\delta_1 q, \dotsc, \delta_4 q \in L^2([0,T],\bbC)$\, of \,$q$\, with only finitely many non-zero
Fourier coefficients, so that the matrix \,$(\eta_j(\delta_k q))_{j,k=1,\dotsc,4}$\, has maximal rank, where the linear forms \,$\eta_j$\, are as in Lemma~\ref{lem:closing-linindep}. 
Let \,$f_j := \delta_j q \in L^2([0,T],\bbC)$\, for \,$j\in \{1,\dotsc, 4\}$\,.
By \cite[Proposition~5.2]{kleinkilian1} there exists \,$N \in \bbN$\, so that
$$ L^2([0,T],\bbC)_{0,N} \to \ell^2(|k|>N),\; \delta q \mapsto (\delta z_k)_{|k|>N} $$
is an isomorphism of Banach spaces; we can choose \,$N$\, large enough so that additionally \,$\widehat{f_j}(k) = 0$\, for all \,$k\in\bbZ$\, with \,$|k|>N$\, and all \,$j\in\{1,\dotsc,4\}$\,.
Then \,$\Psi_{N,f}'(q)$\, is an isomorphism of Banach spaces. 
\end{proof}

%

\section{Simple factor dressing}

It follows from Theorem~\ref{thm:Psi-diffeo} that any potential \,$q$\, which satisfies the closing condition for \,$\bbH^3$\, is the limit of a sequence \,$q_n$\, of finite gap potentials
for which we have \,$\Delta(\mi+\theta)=\pm 2$\, and \,$\Delta'(\mi+\theta)=0$\,. However, we would like to show that the \,$q_n$\, can be chosen to fulfill the ``full'' closing condition
of Lemma~\ref{lem:closing}. The remaining condition for this is that the monodromy of the \,$q_n$\, at the Sym point \,$\theta+\mi$\, is semisimple; note that this condition does not hold
automatically because the Sym point is not on the real line. We will accomplish this by applying \emph{simple factor dressing} to \,$q$\, and to \,$q_n$\,. We prepare the introduction
of simple factor dressing by first describing loop groups over \,$\mathrm{SL}_2(\bbC)$\,, the Birkhoff factorisation for such loop groups and the dressing transformation \cite{BurP:dre, calini-ivey1998} in general. 

For \,$r>0$\, we define 
$$ C_r = \{ \lambda \in \bbC \mid \| \lambda \| = r \}\,, \quad D_r = \{ \lambda \in \bbC \mid \| \lambda \| < r \}\, \quad\text{and}\quad E_r = \{ \lambda \in \bbC \mid \| \lambda \| > r \} \cup \{ \infty \} \,. $$
Next we define the following $r$-loop groups of $\mathrm{SL}_2 (\mathbb{C})$ by 
\begin{align*}
        \Lambda_r  \mathrm{SL}_2 (\mathbb{C}) &= \{ g: C_r \to \mathrm{SL}_2 (\mathbb{C}) \mid \text{\,$g$\, is analytic} \}\,, \\
        \Lambda_r^+  \mathrm{SL}_2 (\mathbb{C}) &= \{ \,g \in \Lambda_r \mathrm{SL}_2 (\mathbb{C}) \mid g \mbox{ extends analytically to } D_r  \}\,, \\
        \Lambda_r^-  \mathrm{SL}_2 (\mathbb{C}) &= \{ \,g \in \Lambda_r \mathrm{SL}_2 (\mathbb{C}) \mid g \mbox{ extends analytically to } E_r  \}\,, \\
        \Lambda_{r,\mathbbm{1}}^-  \mathrm{SL}_2 (\mathbb{C}) &= \{ \,g \in \Lambda_r^- \mathrm{SL}_2 (\mathbb{C}) \mid g(\infty) = \mathbbm{1} \}\,.
\end{align*}

The \emph{Birkhoff factorisation theorem} \cite{PreS, McI} states in this situation that for every $g \in \Lambda_r  \mathrm{SL}_2 (\mathbb{C})$ there exists a unique factorization
\begin{equation*} 
        g = g_+\, D\,g_-
\end{equation*} 
with $g_- \in \Lambda_{r,\mathbbm{1}}^- \mathrm{SL}_2 (\mathbb{C}),\,g_+ \in \Lambda_r^+ \mathrm{SL}_2 (\mathbb{C})$, and $D = \left( \begin{smallmatrix} \lambda^n & 0 \\ 0 & \lambda^{-n} \end{smallmatrix} \right)$\, for some \,$n \in \mathbb{Z}$\,. When $n = 0$ and thus $D = \mathbbm{1}$ holds, then $g$ is said to lie in the \emph{big cell} of $\Lambda_r \mathrm{SL}_2 (\mathbb{C})$.

We wish to consider dressing operations that maintain the reality condition of Equation~\eqref{eq:reality} for extended frames. To ensure this, the dressing factor \,$g \in \Lambda_r^- \mathrm{SL}_2(\bbC)$\,
itself needs to satisfy a corresponding reality condition, namely
$$ g^* = g \quad\text{with}\quad g^*(\lambda) = \overline{g(\bar{\lambda})^t}^{\,-1} \; . $$
We will see that this condition also ensures that \,$g\,F_\lambda$\, lies in the big cell of $\Lambda_r \mathrm{SL}_2 (\mathbb{C})$. We need notations for the sub-loop-groups of elements satisfying
this reality condition, and thus define
\begin{gather*}
  \Lambda_{r,*}  \mathrm{SL}_2 (\mathbb{C}) = \{ g \in \Lambda_{r}  \mathrm{SL}_2 (\mathbb{C})  \mid g^* = g \}\,, \quad
  \Lambda_{r,*}^\pm \mathrm{SL}_2(\bbC) = \Lambda_{r}^\pm \mathrm{SL}_2(\bbC) \cap \Lambda_{r,*} \mathrm{SL}_2(\bbC) \\
  \text{and}\quad \Lambda_{r,*,\one}^- \mathrm{SL}_2(\bbC) = \Lambda_{r,\one}^- \mathrm{SL}_2(\bbC) \cap \Lambda_{r,*} \mathrm{SL}_2(\bbC) \; .
\end{gather*}

\begin{lemma}
  \label{L:bigcell}
  \,$\Lambda_{r,*} \mathrm{SL}_2 (\mathbb{C})$\, is contained in the big cell of \,$\Lambda_{r} \mathrm{SL}_2 (\mathbb{C})$\,. A given \,$g\in \Lambda_{r,*} \mathrm{SL}_2 (\mathbb{C})$\, therefore has a
  unique decomposition
  $$ g = g_+\,g_- \quad\text{with}\quad g_+ \in \Lambda_{r,*}^+ \mathrm{SL}_2 (\mathbb{C}),\; g_- \in \Lambda_{r,*,\mathbbm{1}}^- \mathrm{SL}_2 (\mathbb{C}) \; . $$
\end{lemma}
\begin{proof}
  Let \,$g\in \Lambda_{r,*} \mathrm{SL}_2 (\mathbb{C})$\, be given. We let \,$g=g_+\,D\,g_-$\, with $g_- \in \Lambda_{r,\mathbbm{1}}^- \mathrm{SL}_2 (\mathbb{C}),\,g_+ \in \Lambda_r^+ \mathrm{SL}_2 (\mathbb{C})$, and $D = \left( \begin{smallmatrix} \lambda^n & 0 \\ 0 & \lambda^{-n} \end{smallmatrix} \right)$\, for some \,$n \in \mathbb{Z}$\, be the Birkhoff factorisation of \,$g$\,. We have
  \begin{equation}
  \label{eq:bigcell:GD-pre}
  g_+\,D\,g_- = g = g^* = g_+^* \,D^{-1}\,g_-^*
  \end{equation}
  and therefore
  \begin{equation}
  \label{eq:bigcell:GD}
  g_+^{-1} g_+^* D^{-1} = D g_- g_-^{*^{-1}} \; .
  \end{equation}
  Put $G_+ = g_+^{-1}g_+^*$ and $G_- = g_-g_-^{*-1}$. Clearly $G_+ \in \Lambda_r^+ \mathrm{SL}_2 (\mathbb{C})$ and $G_- \in \Lambda_{r,\mathbbm{1}}^- \mathrm{SL}_2 (\mathbb{C})$. Since the diagonal entries of $G_-$ are positive and real for $\lambda \in \bbR \cap E_r$, we conclude that also $G_+$ has non-zero diagonal entries.

  Now assume \,$n<0$\,. Then the entries of the first column of $G_+ D^{-1}$  are holomorphic functions at $\lambda = 0$, while the entries of the first row of $DG_-$ are holomorphic functions at $\lambda = \infty$. By considering the first diagonal entry of Equation~\eqref{eq:bigcell:GD} we obtain the relation \,$\lambda^{-n} G_+^{(11)} = \lambda^{n} G_-^{(11)}$\, for the first diagonal entries \,$G_\pm^{(11)}$\,
  of \,$G_\pm$\,. This shows $G_+^{(11)} = G_-^{(11)} \equiv 0$, giving the contradiction. The assumption \,$n>0$\, similarly yields a contradiction by considering the second diagonal entries in Equation~\eqref{eq:bigcell:GD}. Therefore we have \,$n=0$\,, which means that \,$g$\, lies in the big cell of \,$\Lambda_{r} \mathrm{SL}_2 (\mathbb{C})$\,. Inserting \,$D=\one$\, in Equation~\eqref{eq:bigcell:GD-pre}
  now gives that \,$g=g_+\,g_- = g_+^*\,g_-^*$\, are ``two'' Birkhoff factorisations of \,$g$\,, so by the uniqueness of the Birkhoff factorisation we have \,$g_\pm^* = g_\pm$\,, hence
  \,$g_\pm \in \Lambda_{r,*}^\pm \mathrm{SL}_2(\bbC)$\,. 
\end{proof}

For any \,$g \in \Lambda_{r,*,\mathbbm{1}}^- \mathrm{SL}_2 (\mathbb{C})$\, and \,$F \in \Lambda_{r,*}^+ \mathrm{SL}_2(\bbC)$\, we have \,$g\,F \in \Lambda_{r,*} \mathrm{SL}_2(\mathbb{C})$\,, and hence
it is a consequence of Lemma~\ref{L:bigcell} that \,$g\,F$\, lies in the big cell of \,$\Lambda_{r} \mathrm{SL}_2(\mathbb{C})$\,. We thus have a unique Birkhoff decomposition
$$ gF = (gF)_+\,(gF)_- \quad\text{with}\quad (gF)_+ \in \Lambda_{r,*}^+ \mathrm{SL}_2 (\mathbb{C}),\; (gF)_- \in \Lambda_{r,*,\mathbbm{1}}^- \mathrm{SL}_2 (\mathbb{C}) \; . $$
The \emph{dressing action} is the group action of \,$\Lambda_{r,*,\one}^-  \mathrm{SL}_2 (\mathbb{C})$\, on \,$\Lambda_{r,*}^+  \mathrm{SL}_2 (\mathbb{C})$\, given by
$$ \Lambda_{r,*,\one}^-  \mathrm{SL}_2 (\mathbb{C}) \times \Lambda_{r,*}^+  \mathrm{SL}_2 (\mathbb{C}) \to \Lambda_{r,*}^+  \mathrm{SL}_2 (\mathbb{C}), \; (g,F) \mapsto g\# F := (g\,F)_+ \; . $$
Let us first verify that this indeed defines an action on extended frames. When the dependence on $r$ is insignificant, we omit it.
\begin{lemma}\label{th:newV} Let $F$ be an extended frame, and $F^{-1}\tfrac{\mathrm{d}F}{\mathrm{d}x} = A_0 + \lambda A_1$ where $A_1 = \tfrac{1}{2}\,\varepsilon$. Let $h \in \Lambda_{*,\mathbbm{1}}^- \mathrm{SL}_2 (\mathbb{C})$, and $hF = G k$ the Birkhoff decomposition, so that $G = h\# F$. Expanding $k = \mathbbm{1} + k_1\lambda^{-1} + O(\lambda^{-2})$, then
\begin{equation} \label{eq:newV}
	G^{-1}\tfrac{\mathrm{d}G}{\mathrm{d}x} = F^{-1}\tfrac{\mathrm{d}F}{\mathrm{d}x} + [k_1,\,A_1]\,.
\end{equation}
\end{lemma}
\begin{proof}
From $G = hF k^{-1}$ we obtain
\[
	G^{-1}\tfrac{\mathrm{d}G}{\mathrm{d}x} = kF^{-1}\tfrac{\mathrm{d}F}{\mathrm{d}x} k^{-1} - \tfrac{\mathrm{d}k}{\mathrm{d}x} \,k^{-1}\,.
\]
By construction $G^{-1} \tfrac{\mathrm{d}G}{\mathrm{d}x}$ has only $\lambda^0$ and $\lambda^1$ terms. Since $dk \,k^{-1} \in O(\lambda^{-1})$, these can only come from 
\[
	kF^{-1}\tfrac{\mathrm{d}F}{\mathrm{d}x} k^{-1} = \left( \mathbbm{1} + k_1\lambda^{-1} + O(\lambda^{-2}) \right) (A_0 + \lambda A_1) \left(\mathbbm{1} - k_1\lambda^{-1} + O(\lambda^{-2}) \right)\,.
\]
Hence the $\lambda^0$ term is $A_0 + [k_1,\,A_1]$, while the $\lambda^1$ term remains $A_1$.
\end{proof}
We will apply dressing to the extended frames of suitable finite gap potentials to obtain potentials whose monodromy at the Sym point is semisimple. The dressing factors
\,$g \in \Lambda_{r,*,\one}^- \mathrm{SL}_2(\bbC)$\,
which we will use for this purpose are of a special kind; they are called \emph{simple factors} \cite{TerU}. We collect the relevant facts about dressing with simple factors in the following lemma:
\begin{lemma} \label{L:dressing} 
Let \,$q\in L^2([0,T],\bbC)$\,, with corresponding extended frame \,$F$\, and monodromy \,$M$\,. Moreover let \,$\lambda_* \in \bbC \setminus \bbR$\,, \,$L \in \mathbb{CP}^1$\, and \,$r>|\lambda_*|$\, be given.
  \begin{itemize}
  \item[(i)]
    Let \,$\pi_L: \bbC^2 \to \bbC^2$\, be the Hermitian projection onto \,$L$\, and \,$\pi_L^\perp := \one - \pi_L$\,.
    Note that \,$\pi_L^\perp$\, is the Hermitian projection onto \,$L^\perp \in \mathbb{CP}^1$\,.
    Then the simple factor corresponding to \,$(\lambda_*,L)$\, is given by
    $$ 	g_{\lambda_*,L}(\lambda) = \tfrac{1}{\sqrt{(1-\lambda_*\,\lambda^{-1})(1-\overline{\lambda_*}\,\lambda^{-1})}} \left[ (1-\lambda_*\,\lambda^{-1})\pi_L + (1-\overline{\lambda_*}\,\lambda^{-1})\pi_L^\perp \right] \; . $$
Then \,$g_{\lambda_*,L} \in \Lambda_{r,*,\one}^-\mathrm{SL}_2(\bbC)$\,,and \,$g_{\lambda_*,L}$\, has at \,$\lambda=\infty$\, the asymptotic expansion
    $$ g_{\lambda_*,L}(\lambda) = \one + \tfrac{\overline{\lambda_*}-\lambda_*}{2}(\pi_L-\pi_L^\perp)\,\lambda^{-1} + O(\lambda^{-2}) $$
    and we have
    $$ g_{\lambda_*,L}^{-1}(\lambda) = \tfrac{1}{\sqrt{(1-\lambda_*\,\lambda^{-1})(1-\overline{\lambda_*}\,\lambda^{-1})}} \left[ (1-\overline{\lambda_*}\,\lambda^{-1})\pi_L + (1-\lambda_*\,\lambda^{-1})\pi_L^\perp \right] = g_{\lambda_*,L^\perp}(\lambda) \; . $$
  \item[(ii)]
    The dressing of the extended frame \,$F$\, with \,$g_{\lambda_*,L}$\, is given by
    $$ g_{\lambda_*,L} \# F = g_{\lambda_*,L}\, F \,g_{\lambda_*,L'}^{-1} \quad\text{with}\quad L' := F(t,\,\lambda_*)^{-1}L \; . $$
    The apparent singularities of \,$g_{\lambda_*,L}\#F$\, at \,$\lambda=\lambda_*,\overline{\lambda_*}$\, are removable.
  \item[(iii)]
    If \,$L$\, is an eigenline of \,$M(\lambda_*)$\, (i.e.~if \,$M(\lambda_*) L = L$\, holds) then the dressed monodromy \,$g_{\lambda_*,L} \# M$\, of \,$g_{\lambda_*,L}\# F$\, is given by
    $$ g_{\lambda_*,L} \# M = g_{\lambda_*,L}\,M\,g_{\lambda_*,L}^{-1} \; . $$
    In particular, the trace functions and hence the spectral curves corresponding to \,$M$\, and to \,$g_{\lambda_*,L} \# M$\, are equal.
  \item[(iv)]
    There exists one and only one potential in \,$L^2([0,T],\bbC)$\, whose extended frame is \,$g_{\lambda_*,L}\# F$\,. We denote this potential by \,$g_{\lambda_*,L} \# q$\, and call it the
    \emph{dressing} of \,$q$\, by \,$g_{\lambda_*,L}$\,. We have \,$g_{\lambda_*,L} \# q = q + c$\,, where \,$c\in L^2([0,T],\bbC)$\, denotes the upper-right entry of the
    Hermitian matrix \,$2\mi(\lambda_*-\overline{\lambda_*})\,\pi_{L'}$\, with \,$L' = F(t,\,\lambda_*)^{-1}L$\,.
  \item[(v)]
    Let \,$q \in L^2([0,T],\bbC)$\,. Then \,$g_{\lambda_*,L}\# q$\, is of finite gap if and only if \,$q$\, is of  finite gap. 
  \end{itemize}
\end{lemma}

\begin{proof}
(i) It follows by direct calculations that the formula for \,$g_{\lambda_*,L}^{-1}$\, is true and that \,$g_{\lambda_*,L}^* = g_{\lambda_*,L}$\, holds. Near \,$\lambda=\infty$\, we have the
  following asymptotic expansion of \,$g_{\lambda_*,L}$\,:
  \begin{equation*} 
  \begin{split}
	g_{\lambda_*,L}(\lambda) &= \frac{1}{\sqrt{(1-\lambda_*\,\lambda^{-1}) (1-\overline{\lambda_*}\,\lambda^{-1})}} \left( (1-\lambda_*\,\lambda^{-1}) \pi_{L} + (1-\overline{\lambda_*}\,\lambda^{-1})\pi^\perp_{L} \right) \\
        &= \left( 1+ \tfrac{\lambda_* +\overline{\lambda_*}}{2} \lambda^{-1} + O(\lambda^{-2}) \right) \left( \mathbbm{1} -\lambda^{-1} ( \lambda_* \pi_{L}+ \overline{\lambda_*} \pi_{L}^\perp ) \right) \\
        &= \mathbbm{1} + \left( \tfrac{\lambda_* +\overline{\lambda_*}}{2}\mathbbm{1} -  \lambda_* \pi_{L} - \overline{\lambda_*} \pi_{L}^\perp  \right) \lambda^{-1} + O(\lambda^{-2})  \\
        &= \mathbbm{1} + \tfrac{\overline{\lambda_*} -\lambda_*}{2} (\pi_{L} - \pi_{L}^\perp ) \lambda^{-1} + O(\lambda^{-2}) \; .
  \end{split}
  \end{equation*}
  Because of \,$r>|\lambda_*|$\,, this asymptotic expansion also shows \,$g_{\lambda_*,L} \in \Lambda_{r,*,\one}^- \mathrm{SL}_2(\bbC)$\,. 
  
(ii) Clearly $g_{L,\lambda_*} \,F\,g^{-1}_{L',\lambda_*}$ satisfies the reality condition, and is holomorphic in $D_r$ away from $\lambda = \lambda_*,\,\overline{\lambda_*}$ where it has at most simple poles. The residues there are
\[
        \Res_{\lambda = \lambda_*}  (g_{L,\lambda_*} \,F\,g^{-1}_{L',\lambda_*}) = 2\lambda_* \,\pi_L^\perp F(t,\,\lambda_*) \pi_{L'} = 0
\]
since $\mathrm{im} (F(t,\,\lambda_*) \pi_{L'} ) \subseteq L$. \\Similarly $\Res_{\lambda = \overline{\lambda_*}}  (g_{L,\overline{\lambda_*}} \,F\,g^{-1}_{L',\overline{\lambda_*}}) = 2\overline{\lambda_*} \,\pi_L F(t,\,\overline{\lambda_*}) \pi_{L'}^\perp = 0$, since $\mathrm{im} (F(t,\,\overline{\lambda_*}) \pi_{L'}^\perp ) \subseteq L^\perp$. Hence $g_{L,\lambda_*} \,F\,g^{-1}_{L',\lambda_*}$ is analytic in $D_r$, which means \,$g_{L,\lambda_*} \,F\,g^{-1}_{L',\lambda_*} \in \Lambda_{r,*}^+ \mathrm{SL}_2(\bbC)$\,. Therefore $g_{L,\lambda_*} F = (g_{L,\lambda_*} \,F\,g^{-1}_{L',\lambda_*})\,(g_{L',\lambda_*})$ is the unique Birkhoff factorization.

(iii) is a direct consequence of (ii).

(iv) We abbreviate \,$g := g_{\lambda_*,L}$\,, \,$G := g\# F$\,, \,$L' = L'(t) := F(t,\,\lambda_*)^{-1}L$\, and \,$h:= g_{\lambda_*,L'} = \one + h_1\,\lambda^{-1} + O(\lambda^{-2})$ with $h_1 = \tfrac{\overline{\lambda_*}-\lambda_*}{2}(\pi_{L'} - \pi_{L'}^\perp)$\,, and the $x$-derivative by a dot. By Lemma~\ref{th:newV} we have 
  \begin{equation}
  \label{eq:dressing:dressed-alpha-2}
  G^{-1}\,\dot{G} = F^{-1}\,\dot{F} + \tfrac12\,[h_1,\,\varepsilon] \; .
  \end{equation}
  Because \,$\varepsilon$\, is diagonal, \,$[h_1,\,\varepsilon]$\, is off-diagonal. Moreover, both \,$h_1$\, and \,$\varepsilon$\, are skew-Hermitian,
  and thus \,$[h_1,\,\varepsilon]$\, is also skew-Hermitian. Therefore we have \,$[h_1,\varepsilon] = c\,\varepsilon_+ + \bar{c}\,\varepsilon_-$\,, where \,$c\in L^2([0,T],\bbC)$\, denotes the
  upper-right entry of the matrix \,$[h_1,\,\varepsilon]$\,. We have
  $$ [h_1,\,\varepsilon] = \tfrac{\overline{\lambda_*}-\lambda_*}{2} [\pi_{L'}-\pi_{L'}^\perp,\,\varepsilon] = (\overline{\lambda_*}-\lambda_*)\,[\pi_{L'},\varepsilon] = \begin{pmatrix} 0 & -2\mi\,\tilde{c} \\ -2\mi\,\bar{\tilde{c}} & 0 \end{pmatrix} \;, $$
  where \,$\tilde{c}$\, denotes the upper-right entry of the matrix \,$\pi_{L'}$\,. It follows that \,$c$\, equals the upper-right entry of the Hermitian matrix \,$2\mi(\lambda_*-\overline{\lambda_*})\,\pi_{L'}$\,.
  Equation~\eqref{eq:dressing:dressed-alpha-2} now shows that
  $$ G^{-1}\,\dot{G} = \tfrac12(\lambda\varepsilon + (q+c)\varepsilon_+ + (\overline{q+c})\varepsilon_-) $$
  holds. This implies that \,$G$\, is the extended frame belonging to the potential \,$q+c \in L^2([0,T],\bbC)$\,, and hence \,$g\# q = q +c$\, holds.

(v) A potential \,$q$\, is of finite gap if and only if the total algebraic genus of the partial normalisation \,$\widetilde{\Sigma}$\, of the spectral curve \,$\Sigma$\, of \,$q$\, on which
  the spectral divisor is locally free, is finite. When we apply the simple factor dressing with \,$g_{\lambda_*,L}$\, to \,$q$\,, the full normalisation \,$\widehat{\Sigma}$\, of \,$q$\,
  remains unchanged and the local \,$\delta$-invariant changes at most at the points above \,$\lambda=\lambda_*$\, and \,$\lambda=\overline{\lambda_*}$\,. From these facts, the statement of (v) follows. 
\end{proof}

In the following lemma we investigate how the order of the zero of \,$M - \tfrac12\Delta\,\one$\, changes under the application of simple factor dressing. These statements
are analogous to those proven by Ehlers-Kn\"orrer in \cite[Section~3, Theorem (i),(iii)]{ehlers-knoerrer} for the B\"acklund transformation for the Korteweg-de Vries equation. 

When \,$\lambda_*\in \bbC$\, is chosen such that \,$\Delta(\lambda_*)^2-4=0$\, holds, then \,$M(\lambda_*)$\, has only a single eigenvalue \,$\mu_* \in \{\pm 1\}$\,.
If \,$M(\lambda_*)$\, is semisimple in this situation, then \,$M(\lambda_*) = \mu_*\,\one$\, holds, so the eigenspace of \,$M(\lambda_*)$\, corresponding to \,$\mu_* = \tfrac12\Delta(\lambda_*)$\,
is 2-dimensional. Despite of this fact, not all eigenvectors of \,$M(\lambda_*)$\, are equivalent. Indeed, if we denote the order of the zero of \,$N := M - \tfrac12\Delta\,\one$\,
at \,$\lambda=\lambda_*$\, by \,$j_0$\,, then \,$\widetilde{N} := \tfrac{1}{(\lambda-\lambda_*)^{j_0}} \,N$\, is holomorphic and non-zero at \,$\lambda=\lambda_*$\,.
Because \,$\widetilde{N}(\lambda_*) \neq 0$\, is nilpotent, its kernel \,$L_*$\, is 1-dimensional. It is distinguished by the fact that it is the line of the holomorphic eigenline bundle of \,$M$\,
on the partial normalisation \,$\widetilde{\Sigma}$\, of \,$\Sigma$\, at \,$(\lambda_*,\mu_*)$\,. Because it also corresponds to the value of the Baker-Akhiezer function associated
to the spectral divisor of \,$M$\, at \,$(\lambda_*,\mu_*)$\,, we call it the \emph{Baker-Akhiezer eigenline} of \,$M(\lambda_*)$\,. 
Part (ii) of the following lemma shows that the Baker-Akhiezer eigenlines also play a special role with respect to simple factor dressing.



\begin{lemma}
  \label{L:dressing:magic}
  Let \,$q\in L^2([0,T],\bbC)$\,, \,$M$\, be the monodromy of \,$q$\,, \,$\Delta = \mathrm{tr} M$\, and \,$\lambda_* \in \bbC\setminus \bbR$\, 
  such that \,$\Delta^2-4$\, has at \,$\lambda=\lambda_*$\, a zero of order \,$n \geq 2$\,. Set  \,$N = M - \tfrac12\Delta\,\one$\,
  and \,$j_0 = \ord_{\lambda_*} N$\,. We have \,$j_0 \leq \lfloor n/2 \rfloor$\,. Moreover let \,$L \in \mathbb{CP}^1$\, be an eigenline of \,$M(\lambda_*)$\,.

  \begin{itemize}
  \item[(i)]
    If \,$j_0\geq 1$\, holds, then \,$\ord_{\lambda_*}(g_{\lambda_*,L} \# M - \tfrac12\Delta\,\one) \geq j_0-1$\,. 
  \item[(ii)]
    If \,$L$\, is the Baker-Akhiezer eigenline of \,$M(\lambda_*)$\,, then we have
    $$ \ord_{\lambda_*}(g_{\lambda_*,L} \# M - \tfrac12\Delta\,\one) = \begin{cases} j_0 + 1 & \text{if \,$j_0 < \lfloor n/2 \rfloor$\,} \\ j_0 & \text{if \,$j_0=\lfloor n/2 \rfloor$\,} \end{cases} \; . $$
  \end{itemize}
\end{lemma}

\begin{proof}
  We begin by noting that Equation~\eqref{eq:closing:Delta2-4-magic} implies \,$j_0 \leq \lfloor n/2 \rfloor$\,.
  We denote by \,$\pi_L$\, the Hermitian projection onto \,$L$\, and set \,$\pi_L^\perp = \one - \pi_L$\,
  as in Lemma~\ref{L:dressing}(i). In the situation of (i), we then have by Lemma~\ref{L:dressing}(iii),(i)
  \begin{align}
    \bigr(g_{\lambda_*,L} \# M \bigr) - \tfrac12\Delta \one
    & = g_{\lambda_*,L}\,M\,g_{\lambda_*,L}^{-1} - \tfrac12\Delta \one = g_{\lambda_*,L}\,N\,g_{\lambda_*,L}^{-1} \notag \\
    \label{eq:dressing:magic:M-pi}
    & = \pi_L N \pi_L + \pi^\perp_L N \pi^\perp_L + r\, \pi_L N \pi^\perp_L + r^{-1}\, \pi_L^\perp N \pi_L
  \end{align}
  with
  \begin{equation}
  \label{eq:dressing:magic:r}
  r = \frac{1-\lambda_*\,\lambda^{-1}}{1-\overline{\lambda_*}\,\lambda^{-1}} \; .
  \end{equation}
Now \,$N$\, has a zero of order \,$j_0$\, at \,$\lambda=\lambda_*$\,, and therefore the terms of the form \,$\pi_L^{(\perp)}\,N\,\pi_L^{(\perp)}$\, all have zeros of order at least \,$j_0$\,.
Further \,$r$\, has a first order zero at \,$\lambda=\lambda_*$\,, therefore it follows from Equation~\eqref{eq:dressing:magic:M-pi} that \,$\bigr(g_{\lambda_*,L} \# M\bigr) - \tfrac12\Delta \one$\, has a zero of order at least \,$j_0-1$\,, completing the proof of (i). 

In the situation of (ii), we note that \,$\det N$\, is the value of the characteristic polynomial \,$\mu^2 - \Delta\,\mu + 1$\, of \,$M$\, at \,$\mu=\tfrac12\Delta$\,, therefore we have
  $$ \det N = \bigr( \tfrac12\Delta \bigr)^2 - \Delta\, \tfrac12\Delta + 1 = \bigr( -\tfrac14 \bigr) \, (\Delta^2-4) \;, $$
  whence it follows that \,$\det N$\, has at \,$\lambda=\lambda_*$\, a zero of order \,$n$\,. 
  We now define the holomorphic, trace-free endomorphism-valued function \,$\widetilde{N} := \tfrac{1}{(\lambda-\lambda_*)^{j_0}}\,N$\,. Here \,$N(\lambda_*)$\, is non-zero
  and nilpotent, hence its kernel is 1-dimensional; it is equal to the Baker-Akhiezer eigenline \,$L$\, of \,$M(\lambda_*)$\,. We have
  $$ \det \widetilde{N} = \frac{1}{(\lambda-\lambda_*)^{2j_0}}\,\det N \;, $$
  hence \,$\det \widetilde{N}$\, has at \,$\lambda=\lambda_*$\, a zero of order \,$n-2j_0$\,. 

  We now consider such coordinates of \,$\bbC^2$\, that \,$L=[1:0]\in\CPone$\, and \,$L^\perp = [0:1] \in \CPone$\,. With respect to such coordinates, the holomorphic function \,$\widetilde{N}$\,
  into the tracefree endomorphisms is represented by a matrix of the form
  $$ \begin{pmatrix} \gamma_{11} & \gamma_{12} \\ \gamma_{21} & -\gamma_{11} \end{pmatrix} $$
  with holomorphic functions \,$\gamma_{11},\gamma_{12},\gamma_{21}$\,. Because \,$\widetilde{N}(\lambda_*)$\, is non-zero and has kernel \,$[1:0]$\,, we have \,$\gamma_{11}(\lambda_*)=\gamma_{21}(\lambda_*)=0$\,
  and \,$\gamma_{12}(\lambda_*)\neq 0$\,. Analogously to Equation~\eqref{eq:dressing:magic:M-pi} we have
  $$ g_{\lambda_*,L}\,\widetilde{N}\,g_{\lambda_*,L}^{-1} = \pi_L \widetilde{N} \pi_L + \pi^\perp_L \widetilde{N} \pi^\perp_L + r\, \pi_L \widetilde{N} \pi^\perp_L + r^{-1}\, \pi_L^\perp \widetilde{N} \pi_L $$
  with $r$\, as in equation~\eqref{eq:dressing:magic:r}, and therefore \,$g_{\lambda_*,L}\,\widetilde{N} \,g_{\lambda_*,L}^{-1}$\, is represented by the matrix
  \begin{equation}
  \label{eq:dressing:magic:tildeN-rep}
  \begin{pmatrix} \gamma_{11} & r\, \gamma_{12} \\ r^{-1}\,\gamma_{21} & -\gamma_{11} \end{pmatrix} \;.
  \end{equation}
  Because \,$r$\, has a first-order zero at \,$\lambda=\lambda_*$\, and \,$\gamma_{21}$\, is zero there, we see from this representation that 
  \,$g_{\lambda_*,L}\,\widetilde{N}\,g_{\lambda_*,L}^{-1}$\, is holomorphic at \,$\lambda=\lambda_*$\,. It follows that
  \begin{equation}
  \label{eq:dressing:magic:dresstilde}  
  g_{\lambda_*,L}\# M - \tfrac12\Delta \one = (\lambda-\lambda_*)^{j_0} \, g_{\lambda_*,L}\,\widetilde{N}\,g_{\lambda_*,L}^{-1}
  \end{equation}
  has a zero of order at least \,$j_0$\, regardless of the value of \,$j_0$\,. In the case \,$j_0 = \lfloor n/2 \rfloor$\, notice that the order of this zero cannot be larger than
  \,$\lfloor n/2 \rfloor=j_0$\,. 

  We now consider the case \,$j_0 < \lfloor n/2 \rfloor$\,. Then the order of the zero of \,$\det \widetilde{N}=-\gamma_{11}^2-\gamma_{12}\,\gamma_{21}$\, is \,$n-2j_0 \geq 2$\,. Because \,$\gamma_{11}^2$\,
  has a zero of order at least \,$2$\,, and \,$\gamma_{12}(\lambda_*)\neq 0$\,, it follows that \,$\gamma_{21}$\, has a zero of order at least \,$2$\,. This fact together with \,$\gamma_{11}(\lambda_*)=0$\,,
  \,$\gamma_{12}(\lambda_*)\neq 0$\, implies that \,$g_{\lambda_*,L}\,\widetilde{N}\,g_{\lambda_*,L}^{-1}$\, has at \,$\lambda=\lambda_*$\, a zero of the exact order \,$1$\,, see the
  representation~\eqref{eq:dressing:magic:tildeN-rep}. It follows by Equation~\eqref{eq:dressing:magic:dresstilde} that \,$g_{\lambda_*,L}\# M - \tfrac12\Delta\one$\,
  has a zero of the exact order \,$j_0+1$\,. 
\end{proof}

\begin{proposition}
\label{prop:closedpot-dense}
The set of potentials of $T$-periodic finite gap curves in $\bbH^3$ with some total torsion \,$\theta T \in \R$\, is $L^2$-dense in the set of potentials of all $T$-periodic curves in $\bbH^3$ with that total torsion.
\end{proposition}
\begin{proof}
  Let \,$q\in L^2([0,T],\bbC)$\, be the potential of a \,$T$-periodic curve in \,$\bbH^3$\,. If \,$q$\, already is of finite gap, then there is nothing to show. So we now suppose that \,$q$\, is not of finite gap.
  We let \,$M$\, be the monodromy of \,$q$\, and \,$\Delta = \tr M$\,.
  Because \,$q$\, satisfies the closing condition for \,$\bbH^3$\, at the Sym point \,$\lambda_* = \mi + \theta$\,, Lemma~\ref{lem:closing} shows
  that \,$\Delta(\lambda_*) = \pm 2$\, and \,$\Delta'(\lambda_*)=0$\, holds, hence \,$\ord_{\lambda_*}(\Delta^2-4) \geq 2$\,. Moreover, Lemma~\ref{lem:closing} shows that \,$M(\lambda_*)$\, is semisimple, so that \,$M(\lambda_*) = \pm\one$\, holds, which shows that \,$N := M - \tfrac12\Delta\one$\, has at \,$\lambda=\lambda_*$\, a zero of order \,$j_0 \geq 1$\,.

  Let \,$L\in \CPone$\, be the Baker-Akhiezer eigenline of \,$M(\lambda_*)$\,. Then we consider the dressed potential \,$\widetilde{q} := g_{\lambda_*,L}^{-1} \# q \in L^2([0,T],\bbC)$\,.
  We have \,$g_{\lambda_*,L}^{-1} = g_{\lambda_*,L^\perp}$\, by Lemma~\ref{L:dressing}(i), and therefore \,$\ord_{\lambda_*}(g_{\lambda_*,L}^{-1}\#M - \tfrac12\Delta\one) \geq j_0-1 \geq 0$\, by
  Lemma~\ref{L:dressing:magic}(i). Further \,$\widetilde{q}$\, is again not of finite gap by Lemma~\ref{L:dressing}(v). Note that the monodromies of \,$q$\, and of \,$\widetilde{q}$\, have the same trace function
  \,$\Delta$\, and the same Baker-Akhiezer eigenline \,$L$\, at \,$\lambda=\lambda_*$\,. 

  We let \,$(z_k)$\, be the perturbed Fourier coefficients of \,$\widetilde{q}$\,. 
  By Theorem~\ref{thm:Psi-diffeo} there exist \,$N \in \bbN$\,, \,$f_1,\dotsc,f_4 \in L^2([0,T],\bbC)$\,, neighborhoods \,$V$\, of \,$\widetilde{q}$\, in \,$L^2([0,T],\bbC)_{\widetilde{q},N}+ \bbR f_1 + \dotsc + \bbR f_4$\, and \,$W$\, of
  \,$\bigr( (z_k)_{|k|>N},\Delta(\lambda_*),0 \bigr) = \Psi_{N,f}(q)$\, in \,$\ell^2(|k|>N) \times \bbC \times \bbC$\, so that \,$\Psi_{N,f}|V: V \to W$\, is a diffeomorphism. 

  For each \,$n\in \bbN$\, with \,$n\geq N$\,, we define a sequence \,$(z^{(n)}_k)_{|k|>N}$\, in \,$\ell^2(|k|>N)$\, by
  $$ z^{(n)}_k := \begin{cases} z_k & \text{if \,$|k|\leq n$\,} \\ 0 & \text{if \,$|k|>n$\,} \end{cases} \; . $$
  In \,$\ell^2(|k|>N) \times \bbC \times \bbC$\,, the sequence \,$\bigr((z^{(n)}_k)_{|k|>N},\Delta(\lambda_*),0\bigr)$\, then converges for \,$n\to\infty$\, to \,$\bigr((z_k)_{|k|>N},\Delta(\lambda_*),0\bigr)=\Psi_{N,f}(\widetilde{q})$\,, therefore there exists \,$N_1\geq N$\, so that we have \,$\bigr((z^{(n)}_k)_{|k|>N},\Delta(\lambda_*),0\bigr) \in W$\, for all
  \,$n>N_1$\,. For such \,$n$\, we put \,$\widetilde{q}_n := (\Phi_N|V)^{-1}\bigr(\,(z^{(n)}_k)_{|k|>N}\,,\Delta(\lambda_*),0\bigr)$\,.
  Because \,$\Psi_{N,f}|V: V \to W$\, is a diffeomorphism, \,$\widetilde{q}_n$\, converges to \,$\widetilde{q}$\, in  \,$L^2([0,T],\bbC)$\,.
  Moreover only finitely many of the perturbed Fourier coefficients of \,$\widetilde{q}_n$\, are non-zero, so \,$\widetilde{q}_n$\, is a finite gap potential.

  Let \,$\widetilde{M}_{n}$\, be the monodromy of \,$\widetilde{q}_n$\, and \,$\Delta_n = \tr \widetilde{M}_{n}$\,. By construction we have
  \,$\Delta_n(\lambda_*) = \Delta(\lambda_*) = \pm 2$\, and \,$\Delta_n'(\lambda_*)=0$\,, hence \,$\ord_{\lambda_*}(\Delta_n^2-4) \geq 2$\,.
  Let \,$L_n\in\CPone$\, be the Baker-Akhiezer eigenline of \,$\widetilde{M}_{n}$\, at \,$\lambda=\lambda_*$\,. 

  We define \,$q_n := g_{\lambda_*,L_n} \# \widetilde{q}_n \in L^2([0,T],\bbC)$\,. Because \,$\widetilde{q}_n$\, is of finite gap, also \,$q_n$\, is of finite gap by Lemma~\ref{L:dressing}(v).
  The monodromy \,$M_{n}$\, of \,$q_n$\, has the same trace function \,$\Delta_n$\, as \,$\widetilde{q}_n$\,. Because of \,$\ord_{\lambda_*}(\Delta_n^2-4) \geq 2$\,,
  Lemma~\ref{L:dressing:magic}(ii) shows that \,$\ord_{\lambda_*}(M_{n}-\tfrac12\Delta_n\one) \geq 1$\, holds, and hence \,$M_{n}(\lambda_*)=\tfrac12\Delta(\lambda_*)\one$\, is semisimple.
  Therefore \,$M_{n}$\, satisfies the closing condition of Lemma~\ref{lem:closing}. Hence \,$q_n$\, is the potential of a finite gap \,$T$-periodic curve in \,$\bbH^3$\, with
  total torsion \,$\theta T$\,.
  
  Because the Baker-Akhiezer eigenline \,$L_n$\, is determined by the function value at \,$(\lambda_*,\mu_*)$\,
  of the Baker-Akhiezer function associated to the spectral divisor of \,$\widetilde{q}_n$\,, and these data depend continuously on the potential \,$q_n$\,, the sequence \,$(L_n)_{n > N_1}$\, converges
  in \,$\CPone$\, to the Baker-Akhiezer eigenline \,$L$\, of \,$\widetilde{q}$\, respectively \,$q$\,. Because the simple factor dressing action depends continuously on its input data,
  and \,$\widetilde{q}_n \longrightarrow \widetilde{q}$\, in \,$L^2([0,T],\bbC)$\,, it follows that
  $$ q_n = g_{\lambda_*,L_n}\# \widetilde{q}_n \;\longrightarrow\, g_{\lambda_*,L}\# \widetilde{q} = g_{\lambda_*,L}\#(g_{\lambda_*,L^\perp} \# q) = g_{\lambda_*,L}\#(g_{\lambda_*,L}^{-1} \# q) = q $$
  in \,$L^2([0,T],\bbC)$\,. 
\end{proof}

\begin{theorem}
\label{T:curvedense}
The set of closed finite gap curves in \,$\bbH^3$\, with respect to the period \,$T$\, is \,$W^{2,2}$-dense in the Sobolev space of all closed \,$W^{2,2}$-curves of length \,$T$\, in \,$\bbH^3$\,. Moreover, near any closed curve \,$\gamma$\, in \,$\bbH^3$\, there are closed finite gap curves with the same total torsion as \,$\gamma$\,. 
\end{theorem}

\begin{proof}
The proof of this theorem is analogous to the proof of \cite[Theorem~6.7]{kleinkilian1} for the case of \,$\bbS^3$\,, where the reference to \cite[Corollary~6.6]{kleinkilian1} is replaced by
a reference to Proposition~\ref{prop:closedpot-dense}.
\end{proof}


\section{Finite-gap curves in the 2-dimensional space forms}

\begin{lemma}
  \label{L:q-realvalued}
  Let \,$q \in L^2([0,T],\bbC)$\, be given, and let \,$(z_k^q)$\, respectively \,$(z_k^{\bar{q}})$\, be the perturbed Fourier coefficients of \,$q$\, respectively of \,$\overline{q}$\,. Then \,$z_k^{\bar{q}} = \overline{z_{-k}^q}$\, holds for all \,$k\in\bbZ$\,. 
\end{lemma}

\begin{proof}
  In this proof we denote the objects associated to \,$q$\, respectively to \,$\overline{q}$\, by the superscript \,${}^q$\, respectively \,${}^{\bar{q}}$\,.
  We have \,$\alpha^{\bar{q}}(\lambda) = - (\alpha^q (-\lambda))^t$\, in Equation~\eqref{eq:frame1}.
  Therefore the solution \,$F^{\bar{q}}$\, of Equation~\eqref{eq:frame1} satisfies \,$F^{\bar{q}}(t,\,\lambda) = (F^{q}(t,\,-\lambda))^{{t}^{-1}}$\,,
  and hence we also have \,$M^{\bar{q}}(\lambda) = M^q(-\lambda)^{t,-1}$\,. By Equation~\eqref{eq:reality},
  $$ M^{\bar{q}}(\lambda) = \overline{M^q(-\overline{\lambda})} $$
  follows.

  Therefore \,$a^{\bar{q}}(\lambda) = d^{\bar{q}}(\lambda)$\, holds for some \,$\lambda\in\C$\, if and only if \,$a^q(-\overline{\lambda}) = d^q(-\overline{\lambda})$\, holds. Therefore
  the sequence \,$(\lambda_k)_{k\in \bbZ}$\, of zeros of \,$a-d$\, (see \cite[Lemma~3.1]{kleinkilian1}) is given by \,$\lambda_k^{\bar{q}} = -\overline{\lambda_{-k}^q}$\,. Moreover,
  the perturbed Fourier coefficients \,$z_k = b(\lambda_k)$\, (see \cite[Definition~3.2]{kleinkilian1}) satisfy
  $$ z_k^{\bar{q}} = b^{\bar{q}}(\lambda_k^{\bar{q}}) = b^{\bar{q}}(-\overline{\lambda_{-k}^q}) = \overline{b^q(\lambda_{-k}^q)} = \overline{z_{-k}^q} \; . $$
\end{proof}

\begin{proposition}
  \label{P:R2S2-pot-dense}
  Let \,$\bbE^2 \in \{\bbR^2,\bbS^2\}$\,. The set of potentials of \,$T$-periodic finite gap curves in \,$\bbE^2$\, is \,$L^2$-dense in the set of potentials of all \,$T$-periodic
  curves in \,$\bbE^2$\,.
\end{proposition}

\begin{proof}
  Let a closed curve \,$\gamma: [0,T]\to \bbE^2$\, be given, which we will also regard as a curve into \,$\bbE^3$\,. Seen in this way, the torsion of \,$\gamma$\, vanishes, and
  therefore the complex curvature \,$q$\, of \,$\gamma$\, is real-valued and we have \,$\theta=0$\,. Therefore Lemma~\ref{L:q-realvalued} shows that the perturbed Fourier coefficients \,$(z_k)$\, of \,$q$\,
  satisfy \,$z_{-k} = \overline{z_k}$\, for all \,$k$\,.

  We now apply the construction of the proof of \cite[Corollary~6.6]{kleinkilian1} 
  to \,$q$\,. Because of \,$\overline{q}=q$\,, we may suppose without loss of generality that the neighborhood \,$V$\, of \,$q$\, on which \,$\Psi_{N,f}$\, is a diffeomorphism is symmetric
  with respect to complex conjugation, i.e.~that for every \,$q_* \in V$\,, also \,$\overline{q_*} \in V$\, holds. 
  The sequence \,$(z_k^{(n)})$\, constructed in that proof has the property that \,$z_{-k}^{(n)} = \overline{z_k^{(n)}}$\, holds for every \,$n$\, and all \,$k\in \bbZ$\,.
  Therefore,
  the potentials \,$q_n \in V$\, and \,$\overline{q_n}\in V$\, have the same perturbed Fourier coefficients \,$(z_k^{(n)})$\, by Lemma~\ref{L:q-realvalued}. Because \,$\Psi_{N,f}|V$\, is injective,
  it follows that \,$q_n = \overline{q_n}$\, holds, i.e.~\,$q_n$\, is real-valued. It follows that the torsion of the curve \,$\gamma_n$\, with the ``complex'' curvature \,$q_n$\, vanishes,
  hence \,$\gamma_n$\, is a curve in \,$\bbE^2$\,. 
\end{proof}

\begin{lemma}
  \label{L:real-dressing}
  Suppose that \,$q \in L^2([0,T],\bbR)$\, is a real-valued potential. Moreover let \,$\lambda_* = \mi$\, and \,$L\in \mathbb{RP}^1$\,. Then \,$g_{\lambda_*,L}\# q$\, is also real-valued. 
\end{lemma}

\begin{proof}
  It suffices to show that the function \,$c\in L^2([0,T],\bbC)$\, from Lemma~\ref{L:dressing}(iv) is real-valued. \,$c$\, is the upper-right entry of the matrix
  \,$2\mi (\lambda_*-\overline{\lambda_*})\,\pi_{L'} = -4\,\pi_{L'}$\,, where \,$L' = F_{\lambda_*}^{-1}\,L$\,. Because \,$q$\, is real-valued, we have \,$F_{-\lambda} = F_\lambda^{t,-1} = \overline{F_{\bar{\lambda}}}$\,
  (see the proof of Lemma~\ref{L:q-realvalued}), and hence \,$\overline{F_{\lambda_*}}=F_{\lambda_*}$\,. The hypothesis \,$L \in \mathbb{RP}^1$\, therefore implies \,$L' \in \mathbb{RP}^1$\,,
  and hence \,$\pi_{L'}$\, is a real-valued matrix. Thus \,$c$\, is real-valued.
\end{proof}

\begin{proposition}
  \label{P:H2-pot-dense}
  The set of potentials of \,$T$-periodic finite gap curves in \,$\bbH^2$\, is \,$L^2$-dense in the set of potentials of all \,$T$-periodic
  curves in \,$\bbH^2$\,.
\end{proposition}

\begin{proof}
  We repeat the proof of Proposition~\ref{prop:closedpot-dense} in the present setting. Because \,$q$\, is real-valued, \,$\theta=0$\, and therefore \,$\lambda_*=\mi$\, holds.
  Also because \,$q$\, is real-valued, the Baker-Akhiezer eigenline \,$L$\, is in \,$\mathbb{RP}^1$\,,
  and therefore \,$\widetilde{q} = g_{\lambda_*,L^\perp} \# q$\, is also real-valued by Lemma~\ref{L:real-dressing}.
  The sequence \,$\widetilde{q}_n$\, converging to \,$\widetilde{q}$\, can be chosen in \,$L^2([0,T],\bbR)$\, by the same argument as in the proof of Proposition~\ref{P:R2S2-pot-dense},
  and with this choice, the eigenlines \,$L_n$\, are again in \,$\mathbb{RP}^1$\,. Therefore \,$q_n = g_{\lambda_*,L_n} \# \widetilde{q}_n$\, is real-valued by another application
  of Lemma~\ref{L:real-dressing}; because it also satisfies the closing condition for \,$\bbH^3$\,, \,$q_n$\, is the potential of a \,$T$-periodic curve in \,$\bbH^2$\,.
  As before, \,$q_n$\, converges to \,$q$\, in \,$L^2([0,T],\bbR)$\,. 
\end{proof}

\begin{theorem}
  \label{T:E2-dense}
  Let \,$\bbE^2 \in \{\bbR^2,\bbS^2,\bbH^2\}$\,. The set of closed finite gap curves in \,$\bbE^2$\, with respect to the period \,$T$\, is \,$W^{2,2}$-dense in the Sobolev space of all
  closed \,$W^{2,2}$-curves of length \,$T$\, in \,$\bbE^2$\,. 
\end{theorem}

\begin{proof}
  The proof is analogous to the proof of \cite[Theorem~6.7]{kleinkilian1}, where the reference to \cite[Corollary~6.6]{kleinkilian1} is replaced by a reference to
  Proposition~\ref{P:R2S2-pot-dense} (for \,$\bbE^2 \in \{\bbR^2,\bbS^2\}$\,) resp.~to Proposition~\ref{P:H2-pot-dense} (for \,$\bbE^2=\bbH^2$\,).
\end{proof}
  
\bibliographystyle{amsplain}

\bibliography{ref}

\def\cprime{$'$}
\providecommand{\bysame}{\leavevmode\hbox to3em{\hrulefill}\thinspace}
\providecommand{\MR}{\relax\ifhmode\unskip\space\fi MR }
\providecommand{\MRhref}[2]{%
  \href{http://www.ams.org/mathscinet-getitem?mr=#1}{#2}
}
\providecommand{\href}[2]{#2}
\begin{thebibliography}{10}

\bibitem{BurP:dre}
F.~E. Burstall and F.~Pedit, \emph{Dressing orbits of harmonic maps}, Duke
  Math. J. \textbf{80} (1995), no.~2, 353--382.

\bibitem{calini-ivey1998}
A.~Calini and T.~Ivey, \emph{B\"acklund transformations and knots of constant
  torsion}, J. Knot Theory Ramifications \textbf{7} (1998), no.~6, 719--746.

\bibitem{caliniivey2001}
\bysame, \emph{Connecting geometry, topology and spectra for finite-gap {NLS}
  potentials}, Phys. D \textbf{152/153} (2001), 9--19, Advances in nonlinear
  mathematics and science.

\bibitem{calini-ivey2005}
\bysame, \emph{Finite-gap solutions of the vortex filament equation: genus one
  solutions and symmetric solutions}, J. Nonlinear Sci. \textbf{15} (2005),
  no.~5, 321--361.

\bibitem{ehlers-knoerrer}
F.~Ehlers and H.~Kn\"orrer, \emph{An algebro-geometric interpretation of the
  {B}\"acklund-transformation for the {K}orteweg-de {V}ries equation}, Comment.
  Math. Helvetici \textbf{57} (1982), 1--10.

\bibitem{farkaskra}
H.~M. Farkas and I.~Kra, \emph{{R}iemann surfaces}, second ed., Graduate Texts
  in Mathematics, vol.~71, Springer-Verlag, New York, 1992. \MR{1139765}

\bibitem{forster}
O.~Forster, \emph{Lectures on {R}iemann surfaces}, Springer, 1981.

\bibitem{goldstein-petrich1991}
R.~E. Goldstein and D.~M. Petrich, \emph{The {K}orteweg-de {V}ries hierarchy as
  dynamics of closed curves in the plane}, Phys. Rev. Lett. \textbf{67} (1991),
  no.~23, 3203--3206.

\bibitem{grinevich2001}
P.~G. Grinevich, \emph{Approximation theorem for the self-focusing nonlinear
  {S}chr\"odinger equation and for the periodic curves in {$\bold R^3$}}, Phys.
  D \textbf{152/153} (2001), 20--27, Advances in nonlinear mathematics and
  science.

\bibitem{grinevich-schmidt-sfb}
P.~G. {Grinevich} and M.~U. {Schmidt}, \emph{Closed curves in
  \,$\mathbb{R}^3$\,: a characterization in terms of curvature and torsion, the
  hasimoto map and periodic solutions of the filament equation}, SFB 288
  Preprint No.~254, 1997.

\bibitem{Hartshorne}
R.~Hartshorne, \emph{Generalized divisors on {G}orenstein curves and a theorem
  of {N}oether}, J. Math. Kyoto Univ. \textbf{26} (1986), 375--386.

\bibitem{hasimoto1972}
R.~Hasimoto, \emph{A soliton on a vortex filament}, J. Fluid. Mech. \textbf{51}
  (1972), 477--485.

\bibitem{klein-habil}
S.~Klein, \emph{A spectral theory for simply periodic solutions of the
  sinh-{G}ordon equation}, Lecture Notes in Mathematics \textbf{2229} (2018).

\bibitem{kleinkilian1}
S.~{Klein} and M.~{Kilian}, \emph{{On closed finite gap curves in spaceforms
  I}}, submitted for publication (2018), 23 pages.

\bibitem{klss2016}
S.~{Klein}, E.~{L\"ubcke}, M.~U. {Schmidt}, and T.~{Simon}, \emph{{Singular
  curves and Baker-Akhiezer functions}}, ArXiv e-prints (2016), 1--39.

\bibitem{langer1999}
J.~Langer, \emph{Recursion in curve geometry}, New York J. Math. \textbf{5}
  (1999), 25--51 (electronic).

\bibitem{McI}
I.~McIntosh, \emph{Global solutions of the elliptic 2d periodic {T}oda
  lattice}, Nonlinearity \textbf{7} (1994), no.~1, 85--108.

\bibitem{PreS}
A.~Pressley and G.~Segal, \emph{Loop groups}, Oxford Science Monographs, Oxford
  Science Publications, 1988.

\bibitem{schmidt1996}
M.~U. Schmidt, \emph{Integrable systems and {R}iemann surfaces of infinite
  genus}, Mem. Amer. Math. Soc. \textbf{122} (1996), no.~581, viii+111.
  \MR{1329944}

\bibitem{schmidtwillmore}
M.~U. {Schmidt}, \emph{{A proof of the Willmore conjecture}}, ArXiv Mathematics
  e-prints (2002), 215 pages.

\bibitem{TerU}
C.~Terng and K.~Uhlenbeck, \emph{B\"{a}cklund transformations and loop group
  actions}, Comm. Pure and Appl. Math \textbf{LIII} (2000), 1--75.

\end{thebibliography}

\end{document}